\documentclass[10pt]{amsart}
\usepackage[utf8]{inputenc}
\usepackage[english]{babel}
\usepackage[T1]{fontenc}
\usepackage{amsmath}
\usepackage{amsfonts}
\usepackage{amssymb}
\usepackage{amsthm}
\usepackage{amscd}
\usepackage{soul}
\usepackage{geometry}
\usepackage{enumitem} 
\usepackage{cancel} 
\usepackage{graphicx}
\usepackage{mathtools}
\usepackage{color}
\usepackage{comment}
\usepackage{yhmath}
\usepackage{color}
\definecolor{marin}{rgb}   {0.,   0.3,   0.7} 
\definecolor{rouge}{rgb}   {0.8,   0.,   0.} 
\definecolor{sepia}{rgb}   {0.8,   0.5,   0.} 
\usepackage[colorlinks,citecolor=marin,linkcolor=rouge,
            bookmarksopen,
            bookmarksnumbered
           ]{hyperref}

\newcommand\N{\mathbb{N}}
\newcommand\T{\mathbb{T}}
\newcommand\Z{\mathbb{Z}}

\newcommand\R{\mathbb{R}}
\newcommand\C{\mathbb{C}}

\newcommand{\dd}{\mathrm{d}}
\newcommand{\enstq}[2]{\left\{#1~\middle|~#2\right\}}

\newcommand\eps{\varepsilon}
\renewcommand{\Re}{\operatorname{Re}}
\renewcommand{\Im}{\operatorname{Im}}

\newcommand\sinc{\mathrm{sinc}}

\newtheorem{theorem}{Theorem}
\newtheorem{proposition}[theorem]{Proposition}
\newtheorem{corollary}[theorem]{Corollary}
\newtheorem{lemma}[theorem]{Lemma}

\title[Growth of Sobolev norms and strong convergence for DNLS]{Growth of Sobolev norms and strong convergence for the discrete nonlinear Schrödinger equation}

\author{Quentin Chauleur}
\address{INRIA Lille, Univ Lille \& Laboratoire Paul Painlevé,
CNRS UMR 8524 Lille, Cité Scientifique, 59655 Villeneuve-d'Ascq, France. }
\email{Quentin.Chauleur@math.cnrs.fr}


\begin{document}

\maketitle

\begin{abstract}
We show the strong convergence in arbitrary Sobolev norms of solutions of the discrete nonlinear Schrödinger on an infinite lattice towards those of the nonlinear Schrödinger equation on the whole space. We restrict our attention to the one and two-dimensional case, with a set of parameters which implies global well-posedness for the continuous equation. Our proof relies on the use of bilinear estimates for the Shannon interpolation as well as the control of the growth of discrete Sobolev norms that we both prove.
\end{abstract}

\section{Introduction}

We consider the discrete nonlinear Schrödinger equation
\begin{equation} \label{DNLS} \tag{DNLS}
i \partial_t u + \Delta_h u = \lambda |u|^{p-1} u,
\end{equation}
where $u : \R \times h\Z^d \rightarrow \C$ with $u(0,\cdot)=u_0$. Here $h>0$ denotes the stepsize of the lattice $h\Z^d$ and 
\[ \Delta_h u(a)=\sum_{j=1}^d \frac{u(a+h e_j)+ u(a-h e_j)-2 u(a)}{h^2}  \]
denotes the discrete Laplace operator for $a \in h\Z^d$ with the canonical basis $(e_j)_{1\leq j \leq d}$ on $\R^d$. We also take $p=2n+1$ with $n\in \N^*$ as an odd integer. For a positive coefficient $\lambda >0$ the equation is called \textit{defocusing}, and \textit{focusing} for $\lambda< 0$.

Equation \eqref{DNLS} has been extensively studied over the past few years, especially in the context of a fixed stepsize $h=1$. A first original result concerning its dynamical properties was obtained by Stefanov and Kevrekidis in \cite{kevrekidis2005}, and reflects weaker dispersion estimates than in the continuous case. In fact, this pathological behavior is induced by resonances of the discrete geometrical setting, or more precisely by the fact that the \textit{symbol} of the discrete operator $\Delta_h$ has critical points and display a lack of convexity, which both appears to be key ingredients for the proof of standard dispersive estimates (we refer to \cite{zuazua2009} for a complete and detailed proof of this feature). Results concerning the existence of standing waves in the focusing case \cite{bambusi2010,jenkinson2016} has naturally followed, and the question of the non-existence of traveling waves, despite long time stability \cite{bernier_traveling}, is still an open question related to a complex phenomenon called the \textit{Peierls-Nabarro barrier} \cite{oxtoby2007}.

From the physical point of view, equation \eqref{DNLS} has been a relevant model both in nonlinear optics, particularly in the context of optical waveguides \cite{eisenberg2002,peschel2002}, or for the modelization of Bose-Einstein condensates trapped in optical lattices formed by laser beams \cite{cataliotti2001, cataliotti2003}. On the other hand, from a mathematical perspective, equation \eqref{DNLS} can also be seen as a first spatial discretization step into the rigorous numerical analysis of the well-known nonlinear Schrödinger equation
\begin{equation} \label{NLS} \tag{NLS}
i \partial_t u + \Delta u = \lambda |u|^{p-1} u
\end{equation}
 on the whole space $\R^d$. In this work, we will in particular be interested by the limit $h \rightarrow 0$ of equation \eqref{DNLS} towards equation \eqref{NLS}, which is usually referred in the literature as the \textit{continuum limit}. This limit has first been investigated in the seminal work of Kirkpatrick, Lenzmann and Staffilani \cite{staffilani2013}, where the authors show the $L^2$ weak convergence of solutions of equation \eqref{DNLS} towards those of \eqref{NLS} using the Banach-Alaoglu theorem as well as discrete Sobolev embedding. The $L^2$ strong convergence of such solutions were then recently achieved by Hong and Yang in \cite{hong2019strong} alongside precise convergence rates in $h$, where the proof is based on uniform discrete Strichartz estimates that the authors had previously shown in \cite{hong2019strichartz} and the use of Gronwall lemma. 
 
Here, we will be interested in the strong convergence in Sobolev spaces $H^s$ for an arbitrary regularity $s \geq 0$. Note that up to the author's knowledge, the only other work dealing with strong Sobolev convergence for continuum limit of such systems is \cite{grande2019continuum} and for the special case $s =\frac12-\frac{1}{2(p-1)}$ with $d=1$, in the context of long-range spatial interactions and time memory effect. Also note that weak and strong $L^2$ convergence for long-range interactions in the one dimensional case are tackle in respectively \cite{staffilani2013} and \cite{hong2019strong}, and has only been generalized very recently to the two dimensional case~\cite{choi2023}.
 
Our approach will follow the strategy of \cite{hong2019strong}, with two main different features. First, we will use the Shannon interpolation (introduced in \cite{bernier_sobolev}) of pointwise discrete function rather than the finite volume type discretization used in \cite{hong2019strong} in order to compare our discrete solution with the continuous one, as it is naturally better suited for Sobolev spaces $H^s$. Secondly, the usual conservation of the $L^2$ norm of the solution of both the discrete equation \eqref{DNLS} and the continuous one \eqref{NLS}, which is broadly used in \cite{hong2019strong}, will be replaced along the proof by the evolution of the Sobolev norms of the solution, which are far from being conserved in the nonlinear setting. In fact, our analysis will require estimates on the growth of Sobolev norms for solutions of both equation \eqref{DNLS} and~\eqref{NLS}. 

On the continuous level, control on the growth of high Sobolev norms of nonlinear dispersive PDEs has been an intensive and still ongoing topic of research, motivated by the study of a nonlinear phenomenon called \textit{weak turbulence}, which basically expects a transfer from low frequencies to high ones (sometimes also referred as \textit{forward cascade}). The literature on this area of research is large, and we should mention, without unrealistically trying to be exhaustive, the seminal works of Bourgain \cite{bourgain1996} and Staffilani \cite{staffilani1997} continued by Sohinger \cite{sohinger2011,sohinger2012} on both periodic and whole space settings, as well as the recent work of Planchon, Tzvetkov and Visciglia \cite{planchon2017} which treats the case of compact manifolds. Note that these results typically provide polynomial or exponential bounds, and are only available in low space dimension $1\leq d \leq 3$ for specific restrictive values of the nonlinearity~$p$ as well as its sign~$\lambda$. One should also mention the particular \textit{integrable} case of the cubic Schrödinger equation (with $p=3$ in equation \eqref{NLS}) in dimension $d=1$, which yields uniform in time estimates for each Sobolev norm $H^m$ with $m \in \N$.

On the other hand, in the discrete framework, only the recent paper of Bernier \cite{bernier_sobolev} actually tackles this problem, in the particular case of a cubic nonlinearity in one dimension. Note that the particular discretization \eqref{DNLS} do not preserves the complete integrability of its continuous counterpart \eqref{NLS} in this setting, at contrary to the usual Ablowitz-Ladik model~\cite{ablowitz2004}.

 Our analysis will be performed on a particular set of parameters for $p$, $d$ and $\lambda$ that we state here:
 \begin{equation} \label{set_parameters}
 \left\{
 \begin{aligned}
 & p=2n+1, \quad n \in \N^* \quad & \text{for} \ \lambda=1 &\ \text{and} \ d=1,2,\\
 & p=3 \quad & \text{for} \ \lambda =-1 &\ \text{and} \ d=1.
 \end{aligned}
 \right.
 \end{equation}
 In particular, we limit our attention to the one and the two-dimensional case for the defocusing case, and only to the cubic nonlinearity in one dimension for the focusing case, where the result of \cite{bernier_sobolev} is available. We also restrict ourselves to odd nonlinearity powers due to technical reasons. Note that the cubic three dimensional case, which was covered in the analysis of \cite{hong2019strong}, will not be handled in our context because of the weak dispersion estimates available in the discrete setting. Some further comments on the three dimensional case will be made in Section \Ref{results_section}. On the other hand, we emphasize that for $d=1,2$ in the defocusing case, we cover the full range of parameters for global well-posedness on the continuous equation \eqref{NLS}. In particular, the global existence of solutions of both \eqref{DNLS} and \eqref{NLS} are guaranteed under the set of parameters \eqref{set_parameters} (see for instance the classical reference~\cite{cazenave}).
 
 This paper is organized as follows. In Section~\ref{results_section}, we will recall some notations and properties of functional analysis on discrete spaces in order to state our main results Theorems~\ref{growth_discrete_sobolev_norms_theorem} and~\ref{convergence_sobolev_DNLS_theorem}. In Section~\ref{shannon_section}, we show some fundamental properties of the Shannon interpolation with respect to discrete Sobobev spaces, with a direct application to the interpolated flow of the discrete linear Schrödinger equation~\eqref{DNLS} taking $\lambda=0$. We then adapt the strategy of the \textit{modified energies} of \cite{planchon2017} to the case of discrete spaces in Section \ref{bounds_sobolev_section} in order to prove Theorem~\ref{growth_discrete_sobolev_norms_theorem}. Gathering all these results, we then conclude by proving Theorem~\ref{convergence_sobolev_DNLS_theorem} in Section~\ref{convergence_section}. We recall some important functional inequalities in both continuous and discrete spaces in Appendix~\ref{appendix_section} which will be used throughout all this paper.

\section{Algebraic context and main results} \label{results_section}

\subsection{Discrete Lebesgue and Sobolev spaces}

We consider a function $g : h\Z^d \rightarrow \C$ with $d \geq 1$. Following \cite{ignat2006}, we denote respectively by $L^p(h\Z^d)$, for $1 \leq p < \infty$, and $L^{\infty}(h \Z^d)$ (or sometimes more compactly $L^p_h$ and $L^{\infty}_h$) the discrete Lebesgue spaces of integrable functions induced by the norms
\[ \| g \|_{L^p(h\Z^d)}^p = h^d \sum_{a \in h\Z^d} |g(a)|^p \quad \text{and} \quad \| g \|_{L^{\infty}(h\Z^d)} = \sup_{a \in h\Z^d} |g(a)|  . \]
Contrary to the continuous case, these spaces are embedded, namely~$L^1(h\Z^d) \subset L^2(h\Z^d) \subset \ldots \subset L^{\infty}(h \Z^d)$, but obviously not uniformly in $h$. In particular, $L^2(h \Z^d)$ is a Hilbert space induced by the scalar product
\[ \langle f,g \rangle_h = h^d \sum_{a \in h \Z^d}  f(a) \overline{g}(a) \]
for $f$, $g :h \Z^d \rightarrow \C$. Denoting respectively the \textit{forward} and \textit{backward difference operators} in the direction $e_j$ for $1\leq j \leq d$ by
\[ \nabla_{h,j}^+ g(a)=\frac{g(a+he_j)-g(a)}{h} \quad \text{and} \quad  \nabla_{h,j}^- g(a)=\frac{g(a)-g(a-he_j)}{h}  \]
for $a \in h \Z^d$, one can  define the discrete Sobolev spaces $W^{1,p}(h\Z^d)$ or $W^{1,p}_h$ for all $1 \leq p < \infty$ by
\[ \| g \|_{W^{1,p}(h\Z^d)}^p = \| g \|_{L^p(h\Z^d)} +\sum_{j=1}^d \|\nabla_{h,j}^+ g \|_{L^p(h\Z^d)}.  \]
We also naturally define the \textit{forward} and \textit{backward discrete gradients} 
\[ \nabla_h^+ = \left( \nabla_{h,1}^+,\ldots,\nabla_{h,d}^+  \right)^{\top} \quad \text{and} \quad \nabla_h^- = \left( \nabla_{h,1}^-,\ldots,\nabla_{h,d}^-  \right)^{\top}.   \]
Using the definition of the discrete Laplace operator $\Delta_h$ given in our introduction, we can also define for all $m \in \N$ the discrete homogeneous and inhomogeneous Sobolev norm 
\[ \| g \|^2_{\dot{H}^m(h\Z^d)}= \langle (-\Delta_h)^m g  ,g \rangle_{h}  \quad \text{and} \quad \| g \|^2_{H^m(h\Z^d)} = \sum_{k=0}^m \| g \|^2_{\dot{H}^k(h\Z^d)}.  \]
We obviously have equivalence between the norms $\| \cdot\|_{W^{1,2}_h}$ and $\| \cdot\|_{H^1_h}$, uniformly in $h$. Furthermore, note that all the above norms are in fact equivalent, and in particular we get that for all $g \in L^2(h\Z^d)$,
\[ \| g \|_{\dot{H}^m(h\Z^d)} \leq \left( \frac{2\sqrt{m}}{h} \right)^d  \| g \|_{L^2(h\Z^d)} \]
for all $m\geq 1$, as a direct consequence of the triangle inequality. However, these bounds are not uniform in $h$, and become trivial at the limit $h \rightarrow 0$.

\subsection{Discrete Fourier Transform}

We now recall the definition of the \textit{discrete Fourier transform} of a function $g \in L^2(h \Z^d)$, namely
\[ \widehat{g} (\xi) = h^d \sum_{a \in h \Z^d} g(a) e^{-ia \cdot \xi},  \]
for $ \xi \in \T_h^d=\R^d / \left(\frac{2 \pi}{h} \Z^d \right)$. In particular we see that the discrete Fourier defines an isometry from $L^2(h\Z^d)$ to $L^2(\T_h^d)$, and that we have an inversion formula: for all $a \in h\Z^d$,
\[ g(a)= \frac{1}{(2 \pi)^d}\int_{\T_h^d} \widehat{g}(\xi) e^{ia \cdot \xi} \dd \xi.   \]
As proved in \cite{stevenson1991}, by the Hilbert scale property, the definition of discrete Sobolev spaces $H^s(h\Z^d)$ can be extended to any real $s \in \R$ through the norm
\[ \| g \|^2_{H^s(h\Z^d)} =  \frac{1}{(2 \pi) ^d} \int_{\T_h^d} \left( 1+\frac{4}{h^2} \sum_{j=1}^d \sin  \left( \frac{h \xi_j }{2}  \right)^2 \right)^s \left| \widehat{g}( \xi ) \right|^2 \dd \xi, \]
which is equivalent to the definition of discrete Sobolev spaces given previously when~$s \in \N$. With our choice of convention, note that we have for the discrete convolution product the property
\[ \widehat{(f \ast_h g)}(\xi) = \widehat{f}(\xi) \widehat{g}(\xi) \quad \text{where} \quad (f \ast_h g)(a)=h^d \sum_{b\in h\Z^d} f(b) g(a-b).   \]

\subsection{Main results}
As $L^2(h\Z^d)$ is a Banach algebra (which is not the case in the continuous setting), Cauchy Lipschitz Theorem can be applied to get the local well-posedness of equation \eqref{DNLS}, and the $L^2(h\Z^d)$ norm is a constant of motion, namely for any $u$ solution of equation \eqref{DNLS} with initial condition $u(0)=u_0$, we have
\begin{equation} \label{mass_conservation_eq}
\| u(t) \|_{L^2(h\Z^d)}=\| u_0 \|_{L^2(h\Z^d)}
\end{equation}
for all times $t\in \R$, which implies global well-posedness in $L^2(h\Z^d)$, and so in all discrete Sobolev spaces $H^s(h\Z^d)$ for $s\in\R$ by norm equivalence. By standard arguments, we can also show that the energy 
\begin{equation} \label{energy_conservation_eq}
E(u):=\frac{1}{2} \| u(t) \|_{\dot{H}^1(h\Z^d)}^2 + \frac{\lambda}{p+1} \| u(t) \|_{L^{p+1}(h\Z^d)}^{p+1}
\end{equation}
is conserved for all times. In view of our set of parameters for $(d,p)$ and by discrete Gagliardo-Nirenberg inequality (see equation \eqref{discrete_gagliardo_nirenberg} in Appendix \ref{appendix_section}), this implies an a priori estimates on $u$ in $H^1(h\Z^d)$, both uniform in time and with respect to the stepsize $h$. We now state one of the main results of this paper, which gives some polynomial bounds on the growth of discrete Sobolev norms of $u$:

\begin{theorem} \label{growth_discrete_sobolev_norms_theorem}
Let $\eps>0$ and $m \in \N^*$. Let $u \in \mathcal{C}(\R;H^m(h\Z^d))$ be the unique solution of \eqref{DNLS} with initial condition $u_0 \in H^m(h\Z^d)$ and the set of parameters $(\lambda,d,p)$ satisfying \eqref{set_parameters}, then 
\begin{equation} \label{bound_growth_discrete_sobolev_norms}
\| u(t) \|_{H^m(h\Z^d)} \leq C \left(1+t^{2(m-1)+\eps} \right),  
\end{equation}
with $C=C(\eps,m,\| u_0 \|_{H^m(h\Z^d)})$.
\end{theorem}
This deserves some comments:
\begin{itemize}
\item This result is a generalization of \cite[Theorem 1.1]{bernier_sobolev} for both higher nonlinearity powers and to the two-dimensional case, in the defocusing setting. However, the time estimate in \cite{bernier_sobolev} is substantially  better ($t^{\frac{m-1}{2}}$ instead of $t^{2(m-1)+\eps}$), and the factor~$C$ only depends there on the $H^1(h\Z^d)$ norm of the initial condition $u_0$, instead of the $\| u_0 \|_{H^m(h\Z^d)}$ needed here in equation \eqref{bound_growth_discrete_sobolev_norms}. These features are mainly due to the algebraic structure of the cubic nonlinearity along with the use of better Sobolev embeddings in the one-dimensional case.
\item As announced above, our proof relies on the use of suitable \textit{modified energies}, a strategy that has proved useful in a variety of contexts. In the framework of growth of Sobolev norms, few results are available both for generic nonlinearities (typically higher than cubic) and for dimensions higher than one. The only general result we are aware of is the recent work~\cite{planchon2017} on continuous compact manifolds, where only weak dispersive estimates are available~\cite{BGT}. This situation is of course reminiscent of our discrete setting, and it is natural to adapt their strategy. However, in the three-dimensional case, for the cubic nonlinearity, their proof relies on a particular Strichartz-type estimate which allows to only bound the $L^{6/5}$ norm of the nonlinearity, enabling the use of Hölder inequalities with low regularity requirement. This estimate, proven in \cite[Proposition 5.4]{tzvetkov2007}, relies on semi-classical time estimates which are at best unclear in our discrete setting, and whose proof is out of scope of the present paper. We left here open the question of higher dimensions for future works, especially in the three-dimensional case where discrete Strichartz estimates are still available (see \cite{hong2019strichartz} or Lemma \ref{discrete_strichartz_lemma} in Appendix \ref{appendix_section}).
\end{itemize}

We now focus on our strong convergence result in the continuum limit $h \rightarrow 0$. In order to compare sequences of $L^2(h\Z^d)$ with integrable functions defined on the whole space, one needs an interpolation method. As introduced in \cite{bernier_sobolev, bernier_traveling}, we rely on the \textit{Shannon interpolation} $\mathcal{S}_h : L^2(h\Z^d) \rightarrow L^2(\R^d)$ defined by
 \[    \mathcal{S}_h u = \mathcal{F}^{-1} \left( \mathbf{1}_{\T_h^d} \widehat{u}  \right),  \]
which allows to extend a sequence into a real function whose Fourier transform is supported in $\T^d_h$, and where $\mathcal{F}$ denotes the usual Fourier transform on $\R^d$ defined by
\[ \mathcal{F} f(\xi) = \int_{\R^d} f(x)e^{-ix\cdot \xi} \dd x,   \]
for all $\xi \in \R^d$ and $f \in L^2(\R^d)$. We also adopt the following convention for the convolution product
\[ f \ast g (x)= \frac{1}{(2\pi)^d} \int_{\R^d} f(y) g(x-y) \dd y \quad \text{so that} \quad \mathcal{F}(fg)= \mathcal{F}f \ast \mathcal{F} g.  \]
 Of course, we also need an operator of projection for continuous integrable functions on discrete functions on the lattice $h\Z^d$. We naturally choose the \textit{pointwise projection} $\Pi_h$ from $H^s(\R^d)$ to $L^2(h \Z^d)$, which is defined for any function $f \in H^s(\R^d)$ with $s>d/2$, by
  \[  \begin{array}{cccc}
  			\Pi_h : & H^s(\R^d) &\rightarrow & L^2(h \Z^d) \\
  					& f & \mapsto &  \begin{array}{cc}  h\Z^d \rightarrow & \C \\
  																	a \mapsto & f(a)    \end{array} \end{array}.   \]
In fact, any function $f \in H^s(\R^d)$ with $s>d/2$ admits a unique continuous representative $\tilde{f} \in \mathcal{C}(\R^d)$, which will still be denoted by $f$ for conciseness purposes. We now state:

\begin{theorem} \label{convergence_sobolev_DNLS_theorem}
Let $\delta>d/2$, $m=\left \lceil{\delta}\right \rceil$ and $\alpha> m + \frac{d}{2}$ with the set of parameters $(\lambda,d,p)$ satisfying \eqref{set_parameters}. Let $\psi \in \mathcal{C}(\R;H^{\alpha}(\R^d))$ be the unique solution of \eqref{NLS} with initial condition $\psi_0 \in H^{\alpha}(\R^d)$, and let $u$ be the unique solution of \eqref{DNLS} with initial condition $u_0=\Pi_h \psi_0$. Let $0\leq s < \delta - d/2 $ and $\eps>0$, then there exists constants $B=B(d,p,s,\delta,\lambda,\|\psi_0\|_{H^{\alpha}(\R^d)})>0$ and $C=C(d,p,s,\delta,\lambda,\eps)>0$ independent of $h$ such that for all $t \geq 0$,
\begin{equation} \label{convergence_sobolev_DNLS}
\| \mathcal{S}_h u (t) - \psi(t) \|_{H^s(\R^d)} \leq C h^{\frac{\delta-s}{2}-\frac{d}{4}}  \left(1+\|\psi_0\|_{H^{\alpha}(\R^d)} \right)^p e^{B t^{2(p-1)(m-1)+1+\eps}}.
\end{equation}
\end{theorem}

One has to compare this result with \cite[Theorem 1.1]{hong2019strong}. In particular, we observe in our case a slight loss of regularity for the initial data, due to our pointwise projection compared to the finite type volume projection adopted in both works \cite{staffilani2013} and \cite{hong2019strong}. We also have an $\eps$-loss in the exponential bound for time, which is a direct consequence of our estimate for the evolution of discrete Sobolev norms of the solution \eqref{bound_growth_discrete_sobolev_norms}. However, our result covers the case of strong convergence for arbitrary Sobolev norm $H^s$, providing enough regularity on the initial data $\psi_0$. Secondly, one can improve the rate of convergence in the stepsize $h$ by assuming more regularity on $\psi_0$, a feature which was not covered in previous works concerning the continuum limit of the discrete nonlinear Schrödinger equation \eqref{DNLS}, and which is usually referred as being a \textit{compatible estimate} in the finite differences literature \cite{berikelashvili2006}.

We expect our strategy of proof to be quite general, and we plan to apply it to other dispersive discrete nonlinear equations in future works. In particular, some properties concerning the Shannon interpolation in the context of Sobolev spaces, although it might be considered being fairly standard, may not be explicitly written in the literature to the best of the author's knowledge, and a small part of this paper is devoted to the rigorous proof of these properties.

Throughout all these notes, $C$ will denote a generic positive constant independent of the underlying parameters, especially with respect to the stepsize parameter $h$. We will specifically denote by $C=C(\alpha)>0$ a constant depending on the parameter $\alpha$.

\section{Shannon interpolation} \label{shannon_section}
Only for this section, $d\geq 1$ is an arbitrary integer. We recall some classical properties of the pointwise projection of a regular enough continuous function on the grid $h\Z^d$, mostly for completeness purposes. We then introduce and prove some useful properties of the Shannon interpolation with respect to both continuous and discrete Sobolev spaces, especially the bilinear estimate of Proposition \ref{aliasing_shannon_prop} that will be fundamental in the proof of Theorem \ref{convergence_sobolev_DNLS_theorem} in Section \ref{convergence_section}.

\subsection{Pointwise projection}
 We first state a property concerning the discrete Fourier transform of the pointwise projection $\Pi_h$, sometimes referred as the \textit{Poisson summation formula}, that will be important in the following, and that we briefly prove in our particular setting for self-completeness:
 \begin{lemma} \label{fourier_projection}
 Let $f \in H^s(\R^d)$ with $s>d/2$, then, for all $\xi \in \T_h^d$,
 \[ \widehat{\Pi_h f} (\xi)= \sum_{k \in \Z^d} \mathcal{F} f \left(\xi + \frac{2 k \pi}{h} \right).   \]
 \end{lemma}
 \begin{proof}
 First, let's note that for all $x \in \R^d$, we have 
\begin{align*}
 \Pi_h f(a) & = \frac{1}{(2\pi)^d} \int_{\R^d} e^{i a \cdot \xi} \mathcal{F} f(\xi) \dd \xi  = \frac{1}{(2\pi)^d} \sum_{k \in \Z^d} \int_{\T_h^d(k)} e^{i a \cdot \xi} \mathcal{F} f (\xi) \dd \xi \\
            & = \frac{1}{(2\pi)^d} \sum_{k \in \Z^d}  \int_{\T_h^d} e^{i a \cdot \left(\xi + \cancel{\frac{2 k \pi}{h}} \right)} \mathcal{F} f \left(\xi + \frac{2 k \pi}{h} \right) \dd \xi  
 \end{align*}
by linear change of variables and periodicity, where $a \in h\Z^d$ and for $k=(k_1,\ldots,k_d) \in \Z^d$,
\[  \T_h^d(k) = \enstq{ x=(x_1,\ldots,x_d) \in \R^d }{ \frac{(2k_j-1)\pi}{h} \leq x_j \leq \frac{(2k_j+1)\pi}{h} \ \text{for all} \ 1 \leq j \leq d } ,\]
with the convention $\T_h^d(0)=\T_h^d$. On the other hand, as $\Pi_h f \in L^2(h\Z^d)$, we know that for all $a \in h \Z^d$,
\[ \Pi_h f (a)=  \frac{1}{(2\pi)^d} \int_{\T_h^d} e^{ia \cdot \xi}  \widehat{\Pi_h f} (\xi) \dd \xi, \]
hence  we get the result by the inverse discrete Fourier formula.
 \end{proof}
 
 We now show the continuity of the pointwise projection with respect to Sobolev spaces:
 
 \begin{lemma} \label{continuity_pointwise_projection}
 Let $f \in H^{\delta}(\R^d)$ with~$\delta > d/2$, then for all $s \geq 0$ such that $\delta-s> \frac{d}{2}$, we have
\[ \| \Pi_h f \|_{H^s(h\Z^d)} \leq \| f \|_{H^s(\R^d)}+ Ch^{\delta-s}\| f \|_{H^{\delta}(\R^d)},  \]
where $C=C(d,s,\sigma)$.
 \end{lemma}
 \begin{proof}
 Defining $g_k(\xi):=\mathcal{F} f \left(\xi + \frac{2 k \pi}{h} \right) $ for all $\xi \in \T_h^d$ and $k \in \Z^d$, from Lemma \ref{fourier_projection} we infer
 \[ \| \Pi_h f \|_{H^s(h\Z^d)} =  \sum_{k \in \Z^d} \left( \frac{1}{(2 \pi) ^d} \int_{\T_h^d} \left( 1+\frac{4}{h^2} \sum_{j=1}^d \sin  \left( \frac{h \xi_j }{2}  \right)^2 \right)^s \left| g_k(\xi) \right|^2 \dd \xi \right)^{\frac12}.  \]
The term for $k=0$ is easily bounded by $\| f \|_{H^s(\R^d)}$, as for all $\xi \in \T_h^d$, 
\begin{equation} \label{bound_sine}
\frac{4}{h^2} \sum_{j=1}^d \sin  \left( \frac{h \xi_j }{2}  \right)^2 \leq 1+|\xi|^2.
\end{equation} 
On the other hand, for $k\neq 0$, we need to estimate the following integral 
 \[ I_k(t):=\int_{\T^d_h} \left( 1+\frac{4}{h^2} \sum_{j=1}^d \sin  \left( \frac{h \xi_j }{2}  \right)^2 \right)^s  \frac{\left(1+ \left| \xi + \frac{2 k \pi}{h} \right|^2  \right)^{\delta-s}}{\left(1+ \left| \xi + \frac{2 k \pi}{h} \right|^2  \right)^{\delta -s}} \left| g_k(\xi)  \right|^2 \dd \xi  \]
 for $\delta \geq s$ yet to be fixed, in particular we need to get a lower bound of $\left| \xi + 2 k \pi/h \right|^2$ for $\xi \in \T_h^d$. As $k \neq 0$, there exists $j_0$ such that
\[ k_{j_0} = \max_{1 \leq j \leq d} |k_j| >0, \quad \text{hence} \quad  2 |k_{j_0}| -1 = |k_{j_0}| + \underbrace{|k_{j_0}| -1}_{\geq 0} \geq |k_{j_0}|. \]
We can then write that
\[ \min_{\xi \in \T^d_h} \left| \xi + \frac{2 k \pi}{h} \right|^2 \geq  \min_{\xi_{j_0} \in \left[-\pi/h,\pi/h \right]} \left| \xi_{j_0} + \frac{2 k_{j_0} \pi}{h} \right|^2 = \frac{\pi^2}{h^2} \left( 2 |k_{j_0}|-1 \right)^2 \geq \frac{\pi^2 |k|^2}{h^2d}, \]
and making the change of variable $\xi \mapsto \xi-2 k \pi/h$ in the integral $I_k$, from the periodicity of the sine function alongside the bound \eqref{bound_sine} we infer the estimate
\[ I_k(t)\lesssim_{d,s,\delta} h^{2(\delta-s)} \left(\sum_{k\neq0} \frac{1}{|k|^{2{(\delta -s)}}}  \right) \| f \|_{H^{\delta}(\R^d)}^2,    \]
where the infinite sum in the right hand-side of this inequality is finite as soon as $\delta - s >d/2$, which gives the result.  
 \end{proof}

 \subsection{Shannon interpolation}
 We now focus on properties of the Shannon interpolation $\mathcal{S}_h$ defined in Section \ref{results_section}. Note that the Shannon interpolation can also be defined by the finite-element type formula
 \[  u \in L^2(h\Z^d) \mapsto \mathcal{S}_h u (x) = \sum_{a \in h\Z^d}  \sinc \left( \frac{x-a}{h} \right) u(a),  \]
where $\sinc$ denotes the cardinal sine function $\sinc(x)=\frac{\sin(x)}{x}$. Also note that $\Pi_h \circ \mathcal{S}_h g = g$, so $\mathcal{S}_h g$ is the only function in $L^2(\R^d)$ with Fourier transform support included in $\T_h^d$ and whose values on $h\Z^d$ are those of $g$. We first state a continuity property:
 
 \begin{lemma}  \label{boundedness_shannon}
Let $u \in H^s(h\Z^d)$ with $s \geq 0$, then
\[ \| u \|_{H^s(h\Z^d)} \leq \| \mathcal{S}_h u \|_{H^s(\R^d)} \leq \left( \frac{\pi}{2}  \right)^s \| u \|_{H^s(h\Z^d)}.  \]
\end{lemma}
\begin{proof}
We recall that
\[ \| \mathcal{S}_h u \|_{H^s(\R^d)}^2= \frac{1}{(2 \pi)^d} \int_{\R^d} \left( 1 + |\xi|^2 \right)^s | \mathcal{F} \circ  \mathcal{F}^{-1} \left( \mathbf{1}_{\T_h^d} \widehat{u} \right)(\xi) |^2 \dd \xi=\frac{1}{(2 \pi)^d} \int_{\T_h^d} \left( 1 + |\xi|^2 \right)^s | \widehat{u}(\xi) |^2 \dd \xi,\]
hence this property is a direct consequence of the following sharp inequality,
\[  \forall \omega \in \left[-\frac{\pi}{2}, \frac{\pi}{2}  \right],\ \ \ \sin(\omega) \leq \omega \leq \frac{\pi}{2} \sin(\omega), \]
and the fact that $\frac{\pi^2}{4} \simeq 2,47 >1$.
\end{proof}

Let now state and prove the following bilinear estimate, as well as a direct corollary, which will be useful in Section \ref{convergence_section}.

\begin{proposition} \label{aliasing_shannon_prop}
Let $f$, $g \in H^{\delta}(h\Z^d)$ with $\delta >d/2$, and let $0\leq s \leq \delta$. Then
\[ \| \mathcal{S}_h (fg) - (\mathcal{S}_h f) (\mathcal{S}_h g) \|_{H^s(\R^d)} \leq C h^{\delta-s} \|\mathcal{S}_h f \|_{H^{\delta}(\R^d)} \|\mathcal{S}_h g \|_{H^{\delta}(\R^d)}. \]
\end{proposition}
\begin{proof}
Let's first note that
\begin{align*}
\widehat{f}\mathbf{1}_{\T_h^d} \ast \widehat{g}\mathbf{1}_{\T_h^d} & = \int_{\T_h^d} \widehat{f} (\eta) \widehat{g}(\xi-\eta)\mathbf{1}_{\T_h^d}(\xi-\eta) \dd \eta \\
& = \left( \int_{\xi + \T_h^d} \widehat{f}(\xi-\eta) \widehat{g}(\eta) \mathbf{1}_{\T_h^d}(\eta) \right) \mathbf{1}_{\T_{h/2}^d},
\end{align*}
where $\T_{h/2}^d = \R^d / \left(\frac{4\pi}{h} \Z^d \right)= \left[-\frac{2\pi}{h},\frac{2\pi}{h} \right]^d$. We now  write that
\begin{align*}
\| \mathcal{S}_h (fg) - (\mathcal{S}_h f) (\mathcal{S}_h g) \|_{H^s(\R^d)}^2 & =  \int_{\R^d} (1+|\xi|^2)^s \left| \mathcal{F} \circ \mathcal{S}_h(fg)(\xi) - \mathcal{F} \left( \mathcal{S}_h f  \mathcal{S}_h g  \right)(\xi) \right|^2 \dd \xi \\
& = \int_{\R^d} (1+|\xi|^2)^s \left| \widehat{f g}(\xi) \mathbf{1}_{\T_h^d}(\xi) - \left( \widehat{f}\mathbf{1}_{\T_h^d} \ast \widehat{g}\mathbf{1}_{\T_h^d}  \right)(\xi)  \right|^2 \dd \xi \\
& = \mathcal{I}_1 + \mathcal{I}_2,
\end{align*}
where
\[ \mathcal{I}_1 := \int_{\T^d_h} (1+|\xi|^2)^s \left| \widehat{f g} (\xi) - \left( \widehat{f}\mathbf{1}_{\T_h^d} \ast \widehat{g}\mathbf{1}_{\T_h^d}  \right)(\xi)  \right|^2 \dd \xi  \]
and
\[ \mathcal{I}_2 := \int_{\T_{h/2}^d \backslash \T_h^d} (1+|\xi|^2)^s \left| \mathcal{F} \left( \mathcal{S}_h f  \mathcal{S}_h g  \right)(\xi)  \right|^2 \dd \xi . \]
We first estimate $\mathcal{I}_1$. Let's note that
\[ \widehat{f g} (\xi)= \widehat{f} \ast_h \widehat{g}(\xi)=\frac{1}{(2\pi)^d}\int_{\T_h^d}  \widehat{f}(\eta) \widehat{g}(\xi-\eta) \dd \eta,  \]
and  
\[ \left( \widehat{f}\mathbf{1}_{\T_h^d} \ast \widehat{g}\mathbf{1}_{\T_h^d}  \right)(\xi) = \frac{1}{(2\pi)^d} \int_{\left\{\eta \in \T_h^d  \right\} \cap \left\{\xi - \eta \in \T_{h/2}^d \backslash \T_h^d  \right\} } \widehat{f}(\eta) \widehat{g}(\xi-\eta) \dd \eta,   \]
so as $\xi \in \T_h^d$,
\[ \widehat{f g} (\xi) - \left( \widehat{f}\mathbf{1}_{\T_h^d} \ast \widehat{g}\mathbf{1}_{\T_h^d}  \right)(\xi) = \frac{1}{(2\pi)^d}\int_{\T_h^d} \widehat{f}(\eta) \widehat{g}(\xi-\eta) \mathbf{1}_{ \T_{h/2}^d \backslash \T_h^d}(\xi-\eta) \dd \eta.  \]
We now decompose $\T_h^d$ as a disjoint union of intervals ($d=1$), squares ($d=2$) or cubes ($d=3$) of length $\pi/h$ (or so on for $d\geq 4$), namely
\[ \T_h^d = \bigsqcup_{l \in \left\{0,1 \right\}^d} K_l \quad \text{with} \quad K_l=\enstq{x \in \T_h^d}{ (l_j-1)\frac{\pi}{h} < x_j \leq l_j\frac{\pi}{h}, 1\leq j \leq d }. \]
We fix $l\in \left\{0,1 \right\}^d$. Let $\xi \in K_l$ and $\eta \in \T_h^d$, and note that 
\[ \eta \in K_l \Leftrightarrow \xi - \eta \in \T_h^d \Leftrightarrow \mathbf{1}_{\xi - \eta \in \T_{h/2}^d \backslash \T_h^d}(\eta)=0. \]
From this remark, we define the set
\[ \mathcal{N}_l := \enstq{l' \in \left\{0,1 \right\}^d}{\xi-\eta \in \T_{h/2}^d\backslash \T_h^d, \xi \in K_l, \eta \in K_{l'}}= \left\{0,1 \right\}^d \backslash l.  \]
We now also fix $l' \in \mathcal{N}_l$, and we assume $\eta \in K_{l'}$. Then there exists a unique $\kappa(l,l') \in \left\{ -1,1 \right\}^d$ such that
\[ \xi - \eta -\kappa(l,l')\frac{2\pi}{h} \in \T_h^d.    \]
Using these notations, we can rewrite $\mathcal{I}_1$ and get the bound
\begin{equation} \label{new_writing_I_1}
 \mathcal{I}_1 \leq C(d) \sum_{l\in \left\{0,1 \right\}^d} \sum_{l' \in \mathcal{N}_l} \underbrace{\int_{K_l} (1+|\xi|^2)^s \left| \int_{K_{l'}} \widehat{f}(\eta) \widehat{g}(\xi-\eta) \mathbf{1}_{ \T_{h/2}^d \backslash \T_h^d}(\xi - \eta) \dd \eta \right|^2 \dd \xi}_{=: \mathcal{I}_{l,l'}} .  
\end{equation}
By the linear change of variables $\xi'=\xi +2 \kappa(l,l')\pi/h$, and as the function $\widehat{g}$ is $\frac{2\pi}{h}$-periodic, we get (dropping the $'$ on $\xi'$) that
\[ \mathcal{I}_{l,l'} = \int_{K_l+\frac{2 \kappa(l,l')\pi}{h}} \left(1+\left| \xi - \frac{2 \kappa(l,l')\pi}{h}\right|^2\right)^s \left| \int_{K_{l'}} \widehat{f}(\eta) \widehat{g}(\xi-\eta) \mathbf{1}_{\T_h^d}(\xi - \eta) \dd \eta  \right|^2 \dd \xi.  \]
We now write that
\[ \left(1+\left| \xi - \frac{2 \kappa(l,l')\pi}{h}\right|^2\right)^s= \frac{\left(1+\left| \xi - \frac{2 \kappa(l,l')\pi}{h}\right|^2\right)^s}{(1+|\xi|^2)^s} \frac{(1+|\xi|^2)^s}{(1+|\xi|^2)^{\delta}} (1+|\xi|^2)^{\delta}.  \]
As $|\kappa(l,l')|\geq 1$, and $K_l \subset \T_h^d$, for every $\xi \in K_l+2 \kappa(l,l')\pi/h$, we have $|\xi|\geq \pi/h$ so
\[  \frac{1}{(1+|\xi|^2)^{\delta-s}} \leq C(s,d) h^{2(\delta-s)}.  \]
On the other hand, as $\xi - 2 \kappa(l,l')\pi/h \in \T_h^d$, we get that
\[ \left| \xi - \frac{2 \kappa(l,l')\pi}{h}\right| \leq |\xi| \quad \text{so} \quad \frac{\left(1+\left| \xi - \frac{2 \kappa(l,l')\pi}{h}\right|^2\right)^s}{(1+|\xi|^2)^s} \leq 1. \]
Combining these inequalities with rough upper bounds on the integration variables $\xi$ and $\eta$, we get the estimate
\[  \mathcal{I}_{l,l'} \leq C h^{2(\delta-s)} \int_{\R^d} (1+|\xi|^2)^{\delta} \left( \int_{\R^d} |\widehat{f}(\eta)| \mathbf{1}_{\T_h^d}(\eta) |\widehat{g}(\xi-\eta)| \mathbf{1}_{\T_h^d}(\xi - \eta) \dd \eta  \right)^2 \dd \xi.  \]
We now conclude by classical arguments. From the classical estimate
\[ (1+|\xi|^2)^{\delta} \leq C \left( (1+|\xi-\eta|^2)^{\delta}+ (1+|\eta|^2)^{\delta} \right)  \]
for $\delta>0$ and $C=C(\delta)>0$, we infer that
\[ \int_{\R^d} (1+|\xi|^2)^{\frac{\delta}{2}} |\widehat{f}(\eta)| \mathbf{1}_{\T_h^d}(\eta) \widehat{g}(\xi-\eta) \mathbf{1}_{\T_h^d}(\xi - \eta)  \dd \eta 	\lesssim | (1+|\cdot|^2)^{\frac{\delta}{2}} \widehat{g} \mathbf{1}_{\T_h^d}| \ast | \widehat{f}\mathbf{1}_{\T_h^d} |+ | \widehat{g}\mathbf{1}_{\T_h^d} | \ast | (1+|\cdot|^2)^{\frac{\delta}{2}} \widehat{f}\mathbf{1}_{\T_h^d}|. \]
Then, integrating over $ \xi \in \R^d$ and using Young's convolution inequality, we get that
\[ \mathcal{I}_{l,l'} \leq C h^{2(\delta-s)} \left( \|\widehat{g} \mathbf{1}_{\T_h^d}\|_{H^{\delta}(\R^d)}^2 \|\widehat{f} \mathbf{1}_{\T_h^d}\|_{L^1(\R^d)}^2 + \|\widehat{g} \mathbf{1}_{\T_h^d}\|_{L^1(\R^d)}^2 \|\widehat{f} \mathbf{1}_{\T_h^d}\|_{H^{\delta}(\R^d)}^2 \right)    \]
We recall that $\widehat{f} \mathbf{1}_{\T_h^d}=\mathcal{F}\circ \mathcal{S}_h f$ and we observe that
\[ \| \mathcal{F}\circ \mathcal{S}_h f \|_{L^1(\R^d)} \leq \| \mathcal{S}_h f \|_{H^{\delta}(\R^d)}   \]
from classical Fourier transform properties as $\delta>d/2$. As we have finite sums in equation \eqref{new_writing_I_1}, we finally get that
\[ \mathcal{I}_1 \leq C h^{2(\delta-s)} \|\mathcal{S}_h f \|_{H^{\delta}(\R^d)}^2 \|\mathcal{S}_h g \|_{H^{\delta}(\R^d)}^2.\]
 On the other hand, we have
\[ \mathcal{I}_2 = \int_{\R^d\backslash\T_h^d} (1+|\xi|^2)^s \left| \mathcal{F} \left( \mathcal{S}_h f  \mathcal{S}_h g  \right)(\xi) \right|^2 \dd \xi \leq C h^{2(\delta -s)} \|\mathcal{S}_h f \|_{H^{\delta}(\R^d)}^2 \|\mathcal{S}_h g \|_{H^{\delta}(\R^d)}^2.	 \]
as $H^{\delta}(\R^d)$ is an algebra providing that $\delta > d/2$, which ends the proof.
\end{proof}

\begin{corollary} \label{subordinate_shannon_prop}
Let $g \in H^{\delta}(h\Z^d)$ with $\delta >d/2$ and $n\in\N^*$. Then
\begin{equation*}
 \| \mathcal{S}_h( g^n) \|_{H^{\delta}(\R^d)} \leq C \| \mathcal{S}_h g \|_{H^{\delta}(\R^d)}^n.  
 \end{equation*}
 where $C=C(\delta,d,n)$. In particular, for $n_1+n_2=n \in \N^*$, we have
 \begin{equation} \label{subordinate_shannon}
 \| \mathcal{S}_h( g^{n_1} \overline{g}^{n_2}) \|_{H^{\delta}(\R^d)} \leq C \| \mathcal{S}_h g \|_{H^{\delta}(\R^d)}^n.  
 \end{equation}
\end{corollary}
\begin{proof}
We prove the result, which is obvious for $n=1$, by induction. Assuming that \eqref{subordinate_shannon} holds for $n \in \N^*$, we compute using Proposition \ref{aliasing_shannon_prop} and the fact that $H^{\delta}(\R^d)$ is an algebra that
\begin{align*}
\| \mathcal{S}_h( g^{n+1}) \|_{H^{\delta}(\R^d)} & \leq \| \mathcal{S}_h(g^{n+1}) - \mathcal{S}_h( g^n)\mathcal{S}_h g \|_{H^{\delta}(\R^d)} + \| \mathcal{S}_h( g^n)\mathcal{S}_h g \|_{H^{\delta}(\R^d)} \\
& \leq C \|\mathcal{S}_h( g^n)\|_{H^{\delta}(\R^d)} \|\mathcal{S}_h g\|_{H^{\delta}(\R^d)} \\
& \leq C \| \mathcal{S}_h g \|_{H^{\delta}(\R^d)}^{n+1}
\end{align*}
by assumption, which gives the result. Equation \eqref{subordinate_shannon} then naturally follows from the fact that for all $x\in \R^d$,
\[ \mathcal{S}_h \overline{g}(x)=\sum_{a\in h\Z^d}\overline{g}(a)\sinc(\pi(x-a)) = \overline{\mathcal{S}_h g}(x). \]
\end{proof}

\subsection{Interpolation of the linear flow}

The next estimate deals with the error made by consequently projecting then interpolating a continuous function $f$ in terms of Sobolev spaces.

\begin{lemma} \label{continuity_shannon}
Let $f \in H^{\delta}(\R^d)$ and $s \geq 0$ such that $\delta -s > d/2$. Then
\[ \| \mathcal{S}_h \circ \Pi_h f - f \|_{H^s(\R^d)} \leq C h^{\delta-s} \| f \|_{H^{\delta}(\R^d)},  \]
where $C=C(d,s,\delta)>0$.
\end{lemma}
\begin{proof}
We write 
\begin{equation*} 
\| \mathcal{S}_h \circ \Pi_h f - f \|_{H^s(\R^d)}^2 
= \int_{\T^d_h} \left(1 + |\xi|^2 \right)^s \left| \widehat{ \Pi_h f} (\xi) - \mathcal{F}f (\xi)  \right|^2 \dd \xi + \int_{\R^d \backslash \T^d_h} \left(1 + |\xi|^2 \right)^s \left| \mathcal{F}f (\xi)  \right|^2 \dd \xi,
\end{equation*}
as the sum of two functions of disjoint supports. For the first integral, we get from Lemma \ref{fourier_projection} that for all $\xi \in \T^d_h$,
\[ \widehat{ \Pi_h f} (\xi) - \mathcal{F}f (\xi) = \sum_{k \neq 0} \mathcal{F} f \left(\xi + \frac{2 k \pi}{h} \right) =: \sum_{k \neq 0} g_k(\xi). \]
Mimicking the proof of Lemma \ref{continuity_pointwise_projection}, the first integral is handle by the estimate 
\[ \int_{\T^d_h} \left(1 + |\xi|^2 \right)^s \left| \widehat{ \Pi_h f} (\xi) - \mathcal{F}f (\xi)  \right|^2 \dd \xi  \lesssim_{d,s,\delta} h^{2(\delta-s)} \left(\sum_{k\neq0} \frac{1}{|k|^{2{(\delta-s)}}}  \right) \| f \|_{H^{\delta}(\R^d)}^2, \]
which is finite as soon as $\delta-s>\frac{d}{2}$. For the second integral, as $1+ |\xi|^2 \geq d \pi^2/h^2$ for $\xi \in \R^d \backslash \T_h^d$, we easily get that
\[ \int_{\R^d \backslash \T^d_h} \left(1 + |\xi|^2 \right)^s \left| \mathcal{F}f (\xi)  \right|^2 \dd \xi \leq C h^{2(\delta-s)} \| f \|_{H^{\delta}(\R^d)}^2.  \]
\end{proof}

As a first consequence, we can get an error estimate between the interpolation of the discrete linear flow and the continuous one:

\begin{proposition} \label{linear_flow_prop}
Let $\psi_0 \in H^{\delta}(\R^d)$ with $\delta>d/2$, and let's denote $u_0 := \Pi_h \psi_0$. Let $s \geq 0$ such that~$\delta-s > d/2$, then for all $t \geq 0$
we have
\[ \| \mathcal{S}_h e^{it\Delta_h}u_0 - e^{it\Delta} \psi_0 \|_{H^s(\R^d)} \leq C h^{\frac{\delta-s}{2}-\frac{d}{4}} (1+t) \| \psi_0 \|_{H^{\delta}(\R^d)}, \]
where $C=C(d,s,\delta)>0$.
\end{proposition}
\begin{proof}
We decompose our analysis on the two following integrals
\begin{align*}
\| \mathcal{S}_h e^{it\Delta_h}u_0 - e^{it\Delta} \psi_0 \|_{H^s(\R^d)} & \leq \| \mathcal{S}_h e^{it\Delta_h}u_0 - e^{it\Delta} \mathcal{S}_h u_0 \|_{H^s(\R^d)} + \| e^{it\Delta} (\mathcal{S}_h u_0 -  \psi_0 ) \|_{H^s(\R^d)} \\
& =: I_1(t) + I_2(t).
\end{align*} 
We first note that
\[  I_1(t)^2= \int_{\T_h^d} (1+|\xi|^2)^s \left| e^{-it \frac{4}{h^2} \sum_{j=1}^d  \sin \left(  \frac{\xi_j h}{2} \right)^2 } \widehat{u_0}(\xi) - e^{-it|\xi|^2} \widehat{u_0}(\xi) \right|^2 \dd \xi.  \]    
Using the fact that for all $\xi \in \T^d_h$, we have 
\[ \left| e^{-it \left( \frac{4}{h^2} \sum_{j=1}^d  \sin \left(  \frac{\xi_j h}{2} \right)^2 - \xi^2 \right) } - 1 \right| \leq t h^2 |\xi|^4, \]
(see for instance \cite[Section 3.2]{ignat2012} or \cite[Lemma 5.7]{hong2019strong}), we infer that, for $\beta>0$ yet to be fixed, 
\begin{align*}
I_1(t)^2 & \leq t h^2 \int_{|\xi|\leq \frac{\pi}{\sqrt{h}}} (1+|\xi|^2)^s |\xi|^4 \left| \widehat{u_0}(\xi) \right|^2 \dd \xi + \int_{\T^d_h \cap \left\{|\xi|> \frac{\pi}{\sqrt{h}}\right\} } \frac{(1+|\xi|^2)^{s+\beta}}{(1+|\xi|^2)^{\beta} }\left| \widehat{u_0}(\xi) \right|^2 \dd \xi \\
& \leq t \frac{h^{2}}{h^{2-\beta}} \int_{|\xi|\leq \frac{\pi}{\sqrt{h}}} (1+|\xi|^2)^{s+\beta} \left| \widehat{u_0}(\xi) \right|^2 \dd \xi + h^{\beta} \int_{\T^d_h \cap \left\{|\xi|> \frac{\pi}{\sqrt{h}}\right\} } (1+|\xi|^2)^{s+\beta}\left| \widehat{u_0}(\xi) \right|^2 \dd \xi
\end{align*}
so we get that
\[  I_1(t) \leq C (1+t) h^{\frac{\beta}{2}} \| u_0\|_{H^{s+\beta}(h\Z^d)} \leq C (1+t) h^{\frac{\beta}{2}} \left(\| \psi_0\|_{H^{s+\beta}(\R^d)}+ h^{\delta-s-\beta} \| \psi_0\|_{H^{\delta}(\R^d)}    \right)  \]
by Lemma \ref{continuity_pointwise_projection}, providing that $\delta -s-\beta>d/2$. On the other hand, from Lemma \ref{continuity_shannon} we get that
\[ I_2(t) =\| \mathcal{S}_h u_0 -  \psi_0 \|_{H^s(\R^d)} \leq C h^{\delta-s} \| \psi_0 \|_{H^{\delta}(\R^d)},  \]
as we have assume that $\delta-s>d/2$. Taking $\beta=\delta-s-d/2$ and combining both estimates, we get the result.
\end{proof}

\section{Bounds on the growth of discrete Sobolev norms}  \label{bounds_sobolev_section}

This section is devoted to the proof of Theorem \ref{growth_discrete_sobolev_norms_theorem}, which strongly relies on the use of \textit{modified energies}. For clearness purposes, we briefly give the general ideas behind this method and its main ingredients. In the context of Schrödinger equations, the idea is to generalize Kato's trick (which basically allows to read the $H^2$ regularity of the solution from an estimation of $\| \partial_t u \|_{L^2}$ using the expression of the equation) to higher orders : rather than deriving $2k$ times in space equation \eqref{DNLS} in order to get an $H^{2k}$ bound, we derive $k$ times in time to infer a bound on $\| \partial_t^k u \|_{L^2_h}$, which is equivalent to $\|u\|_{H^{2k}_h}$ up to a rest term due to nonlinear effects. The idea is then to identify (using the expression of equation \eqref{DNLS}) a quantity $\mathcal{E}_{2k}$ called \textit{higher-order energy} whose leading term is essentially $\| u \|_{H^{2k}_h}^2$. The technical part is then to finely bound $\mathcal{E}_{2k}$ by a combination of discrete Sobolev inequalities and discrete dispersive estimates, in an attempt to get an estimate of the form
\[ \| u(t) \|_{H^{2k}_h}^2  - \| u_0 \|_{H^{2k}_h}^2 \lesssim \| u(t) \|_{L^{\infty}_T H^{2k}_h}^{\alpha}  \]
with $\alpha <2$, which will implies the polynomial growth of $\| u(t) \|_{H^{2k}_h}$ by an iterative argument.

Note that our proof closely follows the one given in \cite{planchon2017} in the context of growth of Sobolev norms for the nonlinear Schrödinger equation \eqref{NLS} on compact manifolds. The main difficulties here come from both weak dispersion estimates in the discrete setting (see \eqref{homogeneous_discrete_strichartz} and \eqref{inhomogeneous_discrete_strichartz} in Appendix \ref{appendix_section}) and the fact that discrete integration by parts, which stands as an essential part of the proofs involving modified energies, are unsymmetrical, namely
\[  \sum_{a \in h \Z^d} \nabla_h^+ f(a) g(a) = - \sum_{a \in h \Z^d}  f(a) \nabla_h^-g(a).\]
We begin our analysis by estimating even discrete Sobolev norms, which is the natural way to proceed in view of the above paragraph. Comments on the generalization to the case of odd Sobolev norms will also follow, as its essentially relies on the same arguments as for the even case, but applied to a different modified energy $\mathcal{E}_{2k+1}$.

Note that in this section we will extensively use the fact that every solution $u$ of \eqref{DNLS} satisfies the estimate
\begin{equation} \label{uniform_bound_H1}
\| u \|_{\mathcal{C}(\R;H^1(h\Z^d))} \leq C, 
\end{equation}
with $C=C(p,\| u _0\|_{H^1(h\Z^d)})$, which holds uniformly with respect to $h$ in view of the energy conservation \eqref{energy_conservation_eq}, the discrete Sobolev embeddings \eqref{discrete_gagliardo_nirenberg} and from our set of parameters $(\lambda,d,p)$ defined by \eqref{set_parameters}. We will also systematically denote discrete spaces $H^s(h\Z^d)$ by the more compact notation $H^s_h$ in equation mode for conciseness purposes, as there is no ambiguity with continuous spaces in this section. The same way, the uniform norm on a time interval $\left[0,T\right[$ will be denoted $\|\cdot\|_{L^{\infty}_T}$.

\subsection{Even Sobolev norms}
We first suppose that $m=2k$ with $k\in \N$. In the spirit of \cite{planchon2017}, we define
\begin{multline} \label{modified_energy}
\mathcal{E}_{2k}(u(t))=\| \partial_t^k u(t)\|^2_{L^2_h}- \sum_{a \in h\Z^d} \left| \partial_t^{k-1} (|u(t,a)|^{p-1}u(t,a)) \right|^2\\
-\frac12 \sum_{a \in h\Z^d} \sum_{j=1}^d \left| \partial_t^{k-1} \nabla_{h,j}^+ ( |u(t,a)|^2) \right|^2  \sum_{\ell=1}^{\frac{p-1}{2}} |u(t,a)|^{p-1-2\ell} |u(t,a+he_j)|^{2\ell-2} .
\end{multline}
We then differentiate $\mathcal{E}_{2k}$ with respect to time, which gives the following proposition:

\begin{proposition} 
Let $u$ be a solution to \eqref{DNLS} with $u_0 \in H^{2k}(h\Z^d)$ and parameters $(\lambda,d,p)$ satisfying~\eqref{set_parameters}, then
\begin{multline} \label{time_derivative_modified_energy}
\frac{\dd}{\dd t}\mathcal{E}_{2k}(u(t)) =  - \frac12 \sum_{a \in h\Z^d} \sum_{j=1}^d \left| \partial_t^{k-1} \nabla^+_{h,j} ( |u(t,a)|^2) \right|^2  \sum_{\ell=1}^{\frac{p-1}{2}} \partial_t \left(|u(t,a)|^{p-1-2\ell} |u(t,a+he_j)|^{2\ell-2} \right) \\
 +\sum_{n=0}^{k-1} c_n \sum_{a \in h\Z^d} \sum_{j=1}^d \partial_t^n \nabla_{h,j}^+(|u(t,a)|^2) \partial_t^{k-1} \nabla_{h,j}^+(|u(t,a)|^2)  \sum_{\ell=1}^{\frac{p-1}{2}} \partial_t^{k-n} \left( |u(t,a)|^{p-1-2 \ell} |u(t,a+h e_j)|^{2 \ell-2} \right) \\
 +  \langle \partial_t^k(|u|^{p-1}),\partial_t^{k-1}(|\nabla_h^+ u|^2) \rangle_h +  \langle \partial_t^k(|u|^{p-1}),\partial_t^{k-1}(|\nabla_h^- u|^2) \rangle_h \\
 + \Re \sum_{n=0}^{k-1} c_n \langle \partial_t^n(|u|^{p-1}) \partial_t^{k-n}u, \partial_t^{k-1}(|u|^{p-1}u) \rangle_h \\
 + \Re \sum_{n=0}^{k-2} c_n \langle \partial_t^k(|u|^{p-1}) \partial_t^{k-n-1}u, \partial_t^n\Delta_h u \rangle_h + \Im \sum_{n=1}^{k-1} c_n \langle \partial_t^n(|u|^{p-1}) \partial_t^{k-n}u, \partial_t^k u \rangle_h \\
=:   \mathcal{J}_1(t)+ \mathcal{J}_2(t)+ \mathcal{J}_3(t) + \mathcal{J}_4(t) + \mathcal{J}_5(t) + \mathcal{J}_6(t) + \mathcal{J}_7(t). 
\end{multline}
where $(c_n)_n$ denotes explicit complex numbers which may change from line to line. 
\end{proposition}
\begin{proof}
We compute, using equation \eqref{DNLS} as well as discrete integration by parts,
\begin{align*}
\frac{\dd }{\dd t} \left( \|\partial_t^k u \|_{L^2_h}^2  \right) & = 2 \Re \langle \partial_t^{k+1} u, \partial_t^k u \rangle_h = 2 \Re \langle \partial_t^k (-\Delta_h u + |u|^{p-1}u), i \partial_t^k u  \rangle_h \\
& =2 \Im \sum_{a \in h\Z^d} \left| \partial_t^k \nabla_h^+ u(a) \right|^2 + 2 \Re \langle \partial_t^k (|u|^{p-1}u), i \partial_t^ku \rangle_h, 
\end{align*}
so the imaginary part simplifies, leading by Leibniz rule to
\begin{multline*}
\frac{\dd }{\dd t} \left( \|\partial_t^k u \|_{L^2_h}^2  \right) \\
=2 \Re \langle \partial_t^k(|u|^{p-1})u,i\partial_t^k u \rangle_h + 2 \Re \langle |u|^{p-1} \partial_t^ku,i\partial_t^k u \rangle_h + \Re \sum_{n=1}^{k-1} \binom{k}{n} \langle \partial_t^n (|u|^{p-1}) \partial_t^{k-n}u,i\partial_t^k u \rangle_h, 
\end{multline*}
so using again equation \eqref{DNLS}, and by simplification,
\begin{multline} \label{equation_1_proof_modified_energy}
\frac{\dd }{\dd t} \left( \|\partial_t^k u \|_{L^2_h}^2  \right) =2 \Re \langle \partial_t^k(|u|^{p-1})u, - \Delta_h \partial_t^{k-1} u \rangle_h +2 \Re \langle \partial_t^k(|u|^{p-1})u,\partial_t^{k-1}(|u|^{p-1}u) \rangle_h \\
+ \Re \sum_{n=1}^{k-1} c_n \langle \partial_t^n (|u|^{p-1}) \partial_t^{k-n}u,i\partial_t^k u \rangle_h \\
=2 \Re \langle \partial_t^k (|u|^{p-1}) u, -\Delta_h \partial_t^{k-1} u \rangle_h + \frac{\dd}{\dd t} \sum_{a \in h \Z^d} \left|  \partial_t^{k-1}(|u(t,a)|^{p-1}u(t,a)) \right|^2 \\
+ \Re \sum_{n=0}^{k-1} c_n \langle \partial_t^n (|u|^{p-1}) \partial_t^{k-n} u,\partial_t^{k-1}(|u|^{p-1}u) \rangle_h + \Re \sum_{n=1}^{k-1} c_n \langle \partial_t^n (|u|^{p-1}) \partial_t^{k-n}u,i\partial_t^k u  \rangle_h.
\end{multline}
Focusing on the first term on the right hand side, as
\[2 \Re \langle \partial_t^k (|u|^{p-1}) u, -\Delta_h \partial_t^{k-1} u \rangle_h = \sum_{a \in h\Z^d} \partial_t^k(|u|^{p-1}) \left( - \overline{u} \Delta_h \partial_t^{k-1}u -u\Delta_h \partial_t^{k-1}\overline{u}  \right), \]
and noticing that
\[ -\overline{u} \Delta_h \partial_t^{k-1}u-u\Delta_h\partial_t^{k-1} \overline{u} = \partial_t^{k-1}(-\overline{u}\Delta_h u -u \Delta_h \overline{u}) + \Re \sum_{n=0}^{k-2} c_n \partial_t^n \Delta_h u \ \partial_t^{k-1-n} \overline{u}  ,\]
we get from the identity
\[ \Delta_h(|u|^2)=u\Delta_h \overline{u} + \overline{u} \Delta_h u + \left|\nabla_h^+u \right|^2 + \left|\nabla_h^- u \right|^2   \]
and integration by parts that 
\begin{multline} \label{equation_2_proof_modified_energy}
2 \Re \langle \partial_t^k (|u|^{p-1}) u, -\Delta_h \partial_t^{k-1} u \rangle_h =\sum_{a \in h\Z^d} \partial_t^k\nabla_h^+(|u|^{p-1}) \cdot \partial_t^{k-1} \nabla_h^+(|u|^2) \\
+\sum_{a \in h\Z^d} \partial_t^k (|u|^{p-1}) \partial_t^{k-1}(|\nabla_h^+u|^2) + \sum_{a \in h\Z^d} \partial_t^k (|u|^{p-1}) \partial_t^{k-1}(|\nabla_h^-u|^2) \\
 + \Re \sum_{n=0}^{k-2}c_n \sum_{a \in h\Z^d} \partial_t^k(|u|^{p-1}) \partial_t^n\Delta_h u \ \partial_t^{k-1-n} \overline{u}.
\end{multline}

 We then infer that from elementary computations, for all $a \in h\Z^d$,
\[ \nabla_h^+ (|u(t,a)|^{p-1})= \left( \nabla_{h,j}^+(|u(t,a)|^2) \sum_{\ell=1}^{\frac{p-1}{2}} |u(t,a)|^{p-1-2\ell} |u(t,a+he_j)|^{2\ell-2} \right)_{1\leq j \leq d}^{\top},  \]
so the first term in the right hand side of \eqref{equation_2_proof_modified_energy} can be written
\begin{multline*}
\sum_{a \in h\Z^d} \partial_t^k\nabla_h^+(|u(t,a)|^{p-1}) \cdot \partial_t^{k-1} \nabla_h^+(|u(t,a)|^2) \\
= \sum_{a \in h\Z^d} \sum_{j=1}^d \partial_t^k \left( \nabla_{h,j}^+(|u(t,a)|^2) \sum_{\ell=1}^{\frac{p-1}{2}} |u(t,a)|^{p-1-2\ell} |u(t,a+he_j)|^{2\ell-2} \right) \partial_t^{k-1} \left(\nabla_{h,j}^+|u(t,a)|^2 \right) \\
= \frac12 \sum_{a \in h\Z^d} \sum_{j=1}^d \partial_t \left|\partial_t^{k-1} \nabla_{h,j}^+ (|u(t,a)|^2)\right|^2 \sum_{\ell =1}^{\frac{p-1}{2}} |u(t,a)|^{p-1-2\ell} |u(t,a+he_j)|^{2\ell-2} \\
+\sum_{n=0}^{k-1} c_n \sum_{a \in h\Z^d} \sum_{j=1}^d \partial_t^n \nabla_{h,j}^+(|u(t,a)|^2) \partial_t^{k-1} \nabla_{h,j}^+(|u(t,a)|^2)  \sum_{\ell=1}^{\frac{p-1}{2}} \partial_t^{k-n} \left( |u(t,a)|^{p-1-2 \ell} |u(t,a+h e_j)|^{2 \ell-2} \right) ,
\end{multline*}
so we conclude by combining this last equatity with equations \eqref{equation_1_proof_modified_energy} and \eqref{equation_2_proof_modified_energy} and using the expression of $\mathcal{E}_{2k}(u(t))$.
\end{proof}

We then prove that the norms $\|\partial_t^k u \|_{L^2_h}$ and $\| u \|_{H^{2k}_h}$ are comparable, so that the leading term in the modified energy $\mathcal{E}_{2k}(u)$ is equivalent to the Sobolev norm $\| u \|_{H^{2k}_h}$:

\begin{proposition}
Let $u$ be a solution to \eqref{DNLS} with $u_0 \in H^{2k}(h\Z^d)$ and parameters $(\lambda,d,p)$ satisfying~\eqref{set_parameters}, then
\begin{equation} \label{equivalence_sobolev_time_derivative_equation}
\| \partial_t^k u - i^k\Delta_h^k u\|_{H^s_h} \leq C \|u \|_{H^{s+2k-1}_h},
\end{equation}
for any $s \geq 0$, where $C=C\left(k,s,\|u\|_{H^1_h}\right)>0$ is independent of $h$.
\end{proposition}
\begin{proof}
 We prove the result by induction on $k$. The case $k=0$ is obvious. Using the expression of equation \eqref{DNLS}, we easily show that for all $\alpha \in \N$, 
\[ \partial_t^{\alpha} u = i^{\alpha} \Delta_h^{\alpha} u + \sum_{n=0}^{\alpha-1} c_n \partial_t^n \Delta_h^{\alpha-n-1}(|u|^{p-1}u)   \]
for suitable constants $c_n \in \C$, so we get that
\[ \| \partial_t^{k+1}u -i^{k+1} \Delta_h^{k+1} u \|_{H^s_h} \lesssim \sum_{n=0}^k \| \partial_t^n(|u|^{p-1}u) \|_{H^{2k-2n+s}_h}  . \]
Expanding time and spaces derivatives, and using the fact that
\[ \sum_{a\in h\Z^d} |u(t,a)| |u(t,a+\beta h)| \leq \frac12 \left(\sum_{a\in h\Z^d} |u(t,a)|^2 +\sum_{a\in h\Z^d}|u(t,a+\beta h)|^2   \right) \leq \|u(t)\|_{L^2_h}^2   \]
for all $\beta \in \Z^d$, we get by Hölder inequality that for all $0 \leq n \leq k$,
\[ \|\partial_t^n (|u|^{p-1}u) \|_{H^{s+2k-2n}_h} \lesssim \prod_{\substack{n_1+\ldots+n_p=j \\ s_1+\ldots+s_p=2k-2n+s}} \|\partial_t^{n_1} u \|_{W^{s_1,2p}_h} \ldots \|\partial_t^{n_p} u \|_{W^{s_p,2p}_h}.  \]
Then using discrete Sobolev injections $H^1(h\Z^d) \subset L^{2p}(h\Z^d)$ as well as the induction hypothesis, we get that for all $1\leq \ell \leq p$, 
\[ \| \partial_t^{n_{\ell}}u \|_{W^{s_{\ell},2p}_h } \lesssim \| \partial_t^{n_{\ell}}u \|_{H^{s_{\ell}+1}_h} \leq \| u \|_{H^{s_{\ell}+1+2n_{\ell}}},  \]
and from interpolation for $\theta_{\ell}(s+2k)=2 n_{\ell}+s_{\ell}$ with $0<\theta_{\ell}<1$
we conclude that 
\[ \| u \|_{H^{s_{\ell}+1+2n_{\ell}}} \leq \| u \|_{H^{s+2k+1}_h}^{\theta_{\ell}} \underbrace{\| u \|_{H^1_h}^{1-\theta_{\ell}} }_{\leq C}. \]
As the above sum on $n$ and the products on $(n_1,\ldots,n_p)$ and $(s_1,\ldots,s_p)$ are finite, we get the result.
\end{proof}

We now prove the following a priori bound:
\begin{proposition}
Let $u$ be a solution to \eqref{DNLS} with $u_0 \in H^{2k}(h\Z^d)$ and parameters $(\lambda,d,p)$ satisfying~\eqref{set_parameters} and $s\in \left\{0,1\right\}$, then for all $0<T\leq 1$ and all $\eps>0$,
\begin{equation} \label{nonlinear_strichartz_estimate}
\|  \partial_t^n u\|_{L^6_T W^{s,4}_h} \lesssim_{d,p,\eps} \| u \|_{L^{\infty}_T H^{2n+s}_h}^{\frac{3}{4}} \| u \|_{L^{\infty}_T H^{2j+s+1}_h}^{\frac{1}{4}} \| u \|_{L^{\infty}_T H^{2j+2}_h}^{\eps(p-1)}.
\end{equation}
\end{proposition}
\begin{proof}
Using discrete Strichartz estimates from Lemma \ref{discrete_strichartz_lemma} on the equation satisfied by $\partial_t^n u$ for any $n \in \N$, we infer that
\begin{align*}
\| \partial_t^n u \|_{L^6_T W^{s,4}_h} & \lesssim  \| \partial_t^n u(0) \|_{H^{s+\frac{1}{4}}_h} + T \| \partial_t^n (|u|^{p-1}u) \|_{L^{\infty}_T H^{s+\frac{1}{4}}_h} \\
& \lesssim \| \partial_t^n u(0) \|_{H^{s}_h}^{\frac{3}{4}}  \| \partial_t^n u(0) \|_{H^{s+1}_h}^{\frac{1}{4}} + T \| \partial_t^n (|u|^{p-1}u) \|_{L^{\infty}_T H^{s}_h}^{\frac{3}{4}} \| \partial_t^n (|u|^{p-1}u) \|_{L^{\infty}_T H^{s}_h}^{\frac{1}{4}}
\end{align*}  
Using equation \eqref{equivalence_sobolev_time_derivative_equation}, we easily handle the first term in the right hand side of the previous inequality, namely
\[ \|\partial_t^n u(0) \|_{H^s_h} \leq \| \partial_t^nu(0) - i^n\Delta_h^n u(0) \|_{H^s_h} +\| \Delta_h^n u(0) \|_{H^s_h} \lesssim \|u_0\|_{H^{s+2n}_h}   \]
as well as 
\[ \|\partial_t^n u(0) \|_{H^s_h} \lesssim \|u_0\|_{H^{s+2n+1}_h}.  \]
On the other hand, expanding the time derivative on the second term of the right hand side, as well as using the discrete bilinear estimate 
\begin{equation} \label{bilinear_estimate_L_infty}
\| f g \|_{L^{\infty}_h} \lesssim_{\eps} \| f \|_{H^1_h}^{1-\eps} \| f \|_{H^2_h}^{\eps}
\end{equation}  
which stands for all $\eps>0$ as $d \leq 2$, we get that
\begin{align*}
 \| \partial_t^n (|u|^{p-1}u) \|_{H^s_h} & \lesssim \sum_{n_1+\ldots+n_p=n} \| \partial_t^{n_1} u \|_{H^s_h} \| \partial_t^{n_2} u \|_{L^{\infty}_h} \ldots \| \partial_t^{n_p} u \|_{L^{\infty}_h} \\
 & \lesssim \sum_{n_1+\ldots+n_p=n} \| \partial_t^{n_1} u \|_{H^s_h} \| \partial_t^{n_2} u \|_{H^1_h}^{1-\eps} \| \partial_t^{n_2} u \|_{H^2_h}^{\eps}\ldots \| \partial_t^{n_p} u \|_{H^1_h}^{1-\eps} \| \partial_t^{n_p} u \|_{H^2_h}^{\eps},
 \end{align*}
which gives using equation \eqref{equivalence_sobolev_time_derivative_equation} that
\[ \| \partial_t^n (|u|^{p-1}u) \|_{H^s_h} \lesssim \sum_{n_1+\ldots+n_p=n} \| u\|_{H^{2n_1+s}_h} \| u\|_{H^{2n_2+1}_h}^{1-\eps}  \| u\|_{H^{2n_2+2}_h}^{\eps} \ldots \| u\|_{H^{2n_p+1}_h}^{1-\eps}  \| u\|_{H^{2n_p+2}_h}^{\eps}. \] 
Denoting 
\[ \theta_1(2n+s)+(1-\theta_1)=2n_1+s \quad \text{and} \quad \theta_1(2n+s)+(1-\theta_1)=2n_1+1    \]
for all $2\leq l \leq p$, we get by interpolation that 
\[ \| \partial_t^n (|u|^{p-1}u) \|_{H^s_h} \lesssim \| u \|_{H^{2n+s}_h}^{\theta_1} \| u \|_{H^1_h}^{1-\theta_1} \| u \|_{H^{2n+s}_h}^{\theta_2(1-\eps)} \| u \|_{H^1_h}^{(1-\theta_2)(1-\eps)} \ldots \| u \|_{H^{2n+s}_h}^{\theta_p(1-\eps)} \| u \|_{H^1_h}^{(1-\theta_p)(1-\eps)} \| u \|_{H^{2n+2}_h}^{\eps(p-1)}  \]
so we get the result noticing that $\sum_{l=1}^p \theta_l=1$ and using the uniform $H^1$ bound on $u$.
\end{proof}

We can now state and prove the main result of this section, which is the following:
\begin{proposition} \label{prop_estimate_modified_energy}
Let $u$ be a solution to \eqref{DNLS} with $u_0 \in H^{2k}(h\Z^d)$ and parameters $(\lambda,d,p)$ satisfying~\eqref{set_parameters}, then for all $0<T\leq 1$ and $\eps>0$ arbitrary small, we have
\[ \left| \mathcal{E}_{2k}(u(T)) \right| \lesssim \left|\mathcal{E}_{2k}(u_0) \right| + T \| u \|_{L^{\infty}_T H^{2k}_h}^{\frac{4k-4}{2k-1}+\eps} + T^{\frac{2}{3}} \| u \|_{L^{\infty}_T H^{2k}_h}^{\frac{4k-\frac{5}{2}}{2k-1}+\eps}.   \]
\end{proposition}
\begin{proof}
Let's first remark that as
\[ \left| \mathcal{E}_{2k}(u(T))\right| \leq \left|\mathcal{E}_{2k}(u_0)\right| + \int_0^T \left|\frac{\dd }{\dd t} \mathcal{E}_{2k}(u(t))\right| \dd t ,  \]
we are brought back to estimate the time integral of the terms $\mathcal{J}_{\ell}$ for $\ell=1,\ldots,7$ in the right hand side of equation \eqref{time_derivative_modified_energy}. From now on we will denote by $\eps>0$ an arbitrary small constant which may change from line to line, and we will extensively used the estimate
\begin{equation} \label{uniform_bound_H^s}
\|u\|_{L^{\infty}_T H^s_h} \lesssim \| u\|_{L^{\infty}_T H^{2k}_h}^{\frac{s-1}{2k-1}}
\end{equation}
for any $1 \leq s \leq 2k$, which immediately follows from an interpolation inequality with the uniform $H^1$ bound \eqref{uniform_bound_H1}. From norm equivalence and Hölder inequality, we infer 
\[ \int_0^T \left|\mathcal{J}_1(t) \right| \dd t \leq  \sum_{k_1+k_2=k-1} \int_0^T \| \partial_t^{k_1} u \|_{W^{1,4}_h}^2 \| \partial_t^{k_2} u \|_{L^{\infty}_h}^2 \| \partial_t u \|_{L^2_h} \|u \|_{L^{\infty}_h}^{p-4} \dd t. \]
For $k_1+k_2=k-1$, beyond the uniform bound $\| u \|_{L^{\infty}_T L^{\infty}_h}^{p-4} \lesssim C$, we easily get that
\[ \| \partial_t u \|_{L^{\infty}_T L^2_h} \lesssim \| u \|_{L^{\infty}_T H^2_h} \]
from \eqref{equivalence_sobolev_time_derivative_equation} as well as 
\[  \| \partial_t^{k_2} u \|_{ L^{\infty}_T L^{\infty}_h}^2 \lesssim \| \partial_t^{k_2} u \|_{L^{\infty}_T H^1_h}^{2(1-\eps)} \| \partial_t^{k_2} u \|_{L^{\infty}_T H^2_h}^{2\eps} \lesssim \| u \|_{L^{\infty}_T H^{2k_2+1}_h}^{2(1-\eps)} \| u \|_{L^{\infty}_T H^{2k_2+2}_h}^{2\eps} \]
using classical interpolation. We use estimate \eqref{nonlinear_strichartz_estimate} in order to bound the last remaining term, namely
\[ \| \partial_t^{k_1} u \|_{L^2_T W^{1,4}_h}^2  \leq T^{\frac{2}{3}} \| \partial_t^{k_1} u \|_{L^6_T W^{1,4}_h}^2 \lesssim T^{\frac{2}{3}} \| u \|_{L^{\infty}_T H^{2k_1+1}_h}^{\frac{3}{2}} \| u \|_{L^{\infty}_T H^{2k_1+2}_h}^{\frac{1}{2}} \| u \|_{L^{\infty}_T H^{2k_1+2}_h}^{\eps}.\]
Combining these bounds with \eqref{uniform_bound_H^s} we get that
\begin{equation} \label{bound_J1}
\int_0^T \left|\mathcal{J}_1(t) \right| \dd t \lesssim T^{\frac{2}{3}} \| u \|_{L^{\infty}_T H^{2k}_h}^{\frac{4k-\frac{5}{2}}{2k-1}+\eps}.
\end{equation}
The same way, from discrete norm equivalence and Hölder inequality, we get that
\[ \int_0^T \left|\mathcal{J}_2(t) \right| \dd t  \leq \sum_{n=0}^{k-1} |c_n| \sum_{\mathcal{A}_{n,k}} \int_0^T \| \partial_t^{n_1} u \|_{L^{\infty}_h} \| \partial_t^{n_2} u \|_{W^{1,4}_h} \| \partial_t^{k_1} u \|_{L^{\infty}_h} \| \partial_t^{k_2} u \|_{W^{1,4}_h} \| \partial_t^{m_1} u \|_{L^2_h} \prod_{\ell=2}^{p-3}  \| \partial_t^{m_{\ell}} u \|_{L^{\infty}_h}\]
 where the set $\mathcal{A}_{n,k}$ is defined for every $n \in \left\{0,\ldots,k-1\right\}$ by
 \[ \mathcal{A}_{n,k} = \enstq{(n_1,n_2,k_1,k_2,m_1,\ldots,m_{p-3})}{n_1+n_2=n, \ k_1+k_2=k-1, \ m_1+\ldots +m_{p-3}=k-n}.  \]
As for $\mathcal{J}_1$, we can estimate each term separately, using the bilinear estimate \eqref{bilinear_estimate_L_infty} to handle $L^{\infty}(h\Z^d)$ norms uniformly in time, and the nonlinear Strichartz estimate \eqref{nonlinear_strichartz_estimate} for $W^{1,4}(h\Z^d)$ norms combined with Hölder inequality for the time integrability, which gives 
\begin{equation}\label{bound_J2}
 \begin{aligned} 
 \int_0^T \left|\mathcal{J}_2(t) \right| \dd t  & \lesssim T^{\frac{2}{3}} \| u \|_{L^{\infty}_T H^{2n_1+1}_h}^{1-\eps} \| u \|_{L^{\infty}_T H^{2n_2+1}_h}^{\frac{3}{4}} \| u \|_{L^{\infty}_T H^{2n_2+2}_h}^{\frac{1}{4}} \| u \|_{L^{\infty}_T H^{2k_1+1}_h}^{1-\eps} \| u \|_{L^{\infty}_T H^{2k_2+1}_h}^{\frac{3}{4}}  \\
 &\quad \quad \times \| u \|_{L^{\infty}_T H^{2k_2+2}_h}^{\frac{1}{4}} \|u\|_{L^{\infty}_T H^{2m_1}_h} \left( \prod_{\ell=2}^{p-3} \|u\|_{L^{\infty}_T H^{2m_{\ell}+1}_h}^{1-\eps} \right) \| u \|_{L^{\infty}_T H^{2k}_h}^{\eps} \\
 & \lesssim T^{\frac{2}{3}}\| u \|_{L^{\infty}_T H^{2k}_h}^{\frac{4k-\frac{5}{2}}{2k-1}+\eps}.
 \end{aligned}
 \end{equation}
 where we have used the estimate \eqref{uniform_bound_H^s} on each separate term in the last step. Focusing on $\mathcal{J}_3$ (or by obvious symmetry on $\mathcal{J}_4$), we infer by Hölder inequality that 
 \[  \int_0^T \left|\mathcal{J}_3(t) \right| \dd t  \leq \sum_{\substack{k_1+k_2=k-1 \\ n_1+\ldots+n_{p-1}=k}} \int_0^T \|\partial_t^{n_1} u \|_{L^2_h} \left( \prod_{\ell=2}^{p-1} \| \partial_t^{n_{\ell}} u \|_{L^{\infty}_h}  \right) \| \partial_t^{k_1} u \|_{W^{1,4}_h} \| \partial_t^{k_2} u \|_{W^{1,4}_h}, \]
 so that estimating each term as above we get that
\begin{equation}\label{bound_J3}
 \begin{aligned}  \int_0^T \left|\mathcal{J}_3(t) \right| \dd t & \lesssim T^{\frac{2}{3}} \| u \|_{L^{\infty}_T H^{2 n_1}_h} \left( \prod_{\ell=2}^{p-1} \| u \|_{L^{\infty}_T H^{2n_{\ell}+1}_h}^{1-\eps}  \right) \| u \|_{L^{\infty}_T H^{2 k_1+1}_h}^{\frac{3}{4}}  \| u \|_{L^{\infty}_T H^{2 k_1+2}_h}^{\frac{1}{4}} \\
& \quad \quad \times \| u \|_{L^{\infty}_T H^{2 k_2+1}_h}^{\frac{3}{4}}  \| u \|_{L^{\infty}_T H^{2 k_2+2}_h}^{\frac{1}{4}} \|u\|_{L^{\infty}_T H^{2k}_h}^{\eps} \\
 & \lesssim T^{\frac{2}{3}}\| u \|_{L^{\infty}_T H^{2k}_h}^{\frac{4k-\frac{5}{2}}{2k-1}+\eps}.
  \end{aligned}
 \end{equation}
 Arguing as above, we also get that 
  \begin{multline}  \label{bound_J5}
 \int_0^T \left|\mathcal{J}_5(t) \right| \dd t  \\
  \leq \sum_{n=0}^{k-1} |c_n| \sum_{\substack{n_1+\ldots+n_{p-1}=n \\ k_1+\ldots+k_p=k-1 \\ n_1=\max(n_1,\ldots,n_{p-1})}} \int_0^T \| \partial_t^{n_1} u \|_{L^2_h} \left( \prod_{\ell=2}^{p-1} \| \partial_t^{n_{\ell}} u\|_{L^{\infty}_h}  \right) \|\partial_t^{k-n} u \|_{L^{\infty}_h}    \|\partial_t^{k_1} u \|_{L^2_h} \left(\prod_{\ell=2}^p \| \partial_t^{k_{\ell}} u\|_{L^{\infty}_h}  \right) \\
  \lesssim T \| u \|_{L^{\infty}_T H^{2 n_1}_h} \left( \prod_{\ell=2}^{p-1} \| u\|_{L^{\infty}_T H^{2n_{\ell}+1}}^{1-\eps}  \right) \| u \|_{L^{\infty}_T H^{2k-2n+1}_h}^{1-\eps} \|u\|_{L^{\infty}_T H^{2k}_h}^{\eps}  \| u \|_{L^{\infty}_T H^{2 k_1}_h} \left( \prod_{\ell=2}^{p} \| u\|_{L^{\infty}_T H^{2k_{\ell}+1}}^{1-\eps}  \right) \\
  \lesssim T\| u \|_{L^{\infty}_T H^{2k}_h}^{\frac{4k-\frac{5}{2}}{2k-1}+\eps}.
  \end{multline}
 Note that as no $W^{1,4}(h\Z^d)$ is involved in the above estimate, we do not trade space regularity for time integrability by Strichartz inequality \eqref{nonlinear_strichartz_estimate}, and we end up having a straight linear bound with respect to $T$. We end up our analysis by estimating $\mathcal{J}_6$ and $\mathcal{J}_7$ similarly to $\mathcal{J}_1$, $\mathcal{J}_2$, $\mathcal{J}_3$ and $\mathcal{J}_4$, namely
 \begin{equation}\label{bound_J6}
 \begin{aligned}  \int_0^T \left|\mathcal{J}_6(t) \right| \dd t & \leq \sum_{n=0}^{k-2} |c_n| \sum_{ \substack{k_1+\ldots+k_{p-1}=k \\ k_1=\max(k_1,\ldots,k_{p-1})}} \int_0^T \| \partial_t^{k_1} u \|_{L^2_h} \left( \prod_{\ell=2}^{p-1} \| \partial_t^{k_{\ell}} u \|_{L^{\infty}_h} \right) \| \partial_t^n  \Delta_h u \|_{L^4_h} \| \partial_t^{k-1-n} u \|_{L^4_h} \\ 
 & \lesssim T^{\frac{2}{3}} \| u \|_{L^{\infty}_T H^{2 k_1}_h} \left( \prod_{\ell=2}^{p-1} \| u \|_{L^{\infty}_T H^{2k_{\ell}+1}_h}^{1-\eps}  \right) \| u \|_{L^{\infty}_T H^{2 n+2}_h}^{\frac{3}{4}}  \| u \|_{L^{\infty}_T H^{2 n+3}_h}^{\frac{1}{4}} \\
& \quad \quad \times \| u \|_{L^{\infty}_T H^{2k-2n-2}_h}^{\frac{3}{4}}  \| u \|_{L^{\infty}_T H^{2 k-2n-1}_h}^{\frac{1}{4}} \|u\|_{L^{\infty}_T H^{2k}_h}^{\eps} \\
 & \lesssim T^{\frac{2}{3}}\| u \|_{L^{\infty}_T H^{2k}_h}^{\frac{4k-\frac{5}{2}}{2k-1}+\eps}
  \end{aligned}
 \end{equation}
 and 
 \begin{multline}\label{bound_J7}  
 \int_0^T \left|\mathcal{J}_7(t) \right| \dd t  \leq \sum_{n=0}^{k-1} |c_n| \sum_{ n_1+\ldots+n_{p-1}=n } \int_0^T \| \partial_t^{n_1} u \|_{L^4_h} \left( \prod_{\ell=2}^{p-1} \| \partial_t^{n_{\ell}} u \|_{L^{\infty}_h} \right) \| \partial_t^{k-n}  u \|_{L^4_h} \| \partial_t^k\overline{u} \|_{L^2_h} \\ 
  \lesssim T^{\frac{2}{3}} \| u \|_{L^{\infty}_T H^{2 n_1}_h}^{\frac{3}{4}} \| u \|_{L^{\infty}_T H^{2 n_1+1}_h}^{\frac{1}{4}} \left( \prod_{\ell=2}^{p-1} \| u \|_{L^{\infty}_T H^{2n_{\ell}+1}_h}^{1-\eps}  \right) \| u \|_{L^{\infty}_T H^{2k-2n}_h}^{\frac{3}{4}}  \| u \|_{L^{\infty}_T H^{2 k-2n+1}_h}^{\frac{1}{4}}\|u\|_{L^{\infty}_T H^{2k}_h}^{1+\eps} \\
  \lesssim T^{\frac{2}{3}}\| u \|_{L^{\infty}_T H^{2k}_h}^{\frac{4k-\frac{5}{2}}{2k-1}+\eps}.
 \end{multline}
 Finally, summing estimates \eqref{bound_J1}, \eqref{bound_J2}, \eqref{bound_J3}, \eqref{bound_J5}, \eqref{bound_J6} and \eqref{bound_J7} ends the proof.
\end{proof}

We then deduce the key estimate in order to prove Theorem \ref{growth_discrete_sobolev_norms_theorem}:

\begin{proposition}
Let $u$ be a solution to \eqref{DNLS} with $u_0 \in H^{2k}(h\Z^d)$ and parameters $(\lambda,d,p)$ satisfying~\eqref{set_parameters}, then for all $0<T\leq 1$ and $\eps>0$ arbitrary small, we have
\begin{equation} \label{control_H_2k_norm_equation}
\| u(T) \|_{H^{2k}_h}^2-\| u_0 \|_{H^{2k}_h}^2 \lesssim \| u \|_{L^{\infty}_T H^{2k}_h}^{\frac{4k-4}{2k-1}+\eps} + T^{\frac{2}{3}} \| u \|_{L^{\infty}_T H^{2k}_h}^{\frac{4k-\frac{5}{2}}{2k-1}+\eps}.
\end{equation}
\end{proposition}
\begin{proof}
Writing $\mathcal{R}_{2k}(u):=\mathcal{E}_{2k}(u)-\| \partial_t^k u\|_{L^2_h}$ using equation \eqref{modified_energy}, we first focus on estimating both terms of $\mathcal{R}_{2k}(u)$ in the same way than in the proof of Proposition \ref{prop_estimate_modified_energy}: from norm equivalence and equation \eqref{uniform_bound_H^s}, we get that
\begin{multline*}
\sum_{a \in h\Z^d} \sum_{j=1}^d \left| \partial_t^{k-1} \nabla_{h,j}^+ ( |u(t,a)|^2) \right|^2  \sum_{\ell=1}^{\frac{p-1}{2}} |u(t,a)|^{p-1-2\ell} |u(t,a+he_j)|^{2\ell-2} \\
\lesssim \sum_{k_1+k_2=k-1} \| \partial_t^{k_1} u \|_{H^{1}_h}^2 \| \partial_t^{k_2} u \|_{L^{\infty}_h}^2 \| u \|_{L^{\infty}_h}^{p-3} \lesssim \sum_{k_1+k_2=k-1}  \| u \|_{H^{2 k_1+1}_h}^2 \| u \|_{H^{2k_2+1}_h}^2 \| u \|_{H^{2k}_h}^{\eps} \lesssim \| u \|_{H^{2k}_h}^{\frac{4k-4}{2k-1}+\eps}
\end{multline*}
and
\begin{align*}
\sum_{a \in h\Z^d} \left| \partial_t^{k-1} (|u(t,a)|^{p-1}u(t,a)) \right|^2 & \lesssim \sum_{k_1+\ldots+k_p=k-1} \|\partial_t^{n_1} u \|_{L^2_h}^2 \left( \prod_{\ell=2}^p \| \partial_t^{n_{\ell}} u \|_{L^{\infty}_h}^2  \right) \\
& \lesssim \sum_{n_1+\ldots+n_p=k-1} \| u \|_{H^{2n_1}_h}^2 \left( \prod_{\ell=2}^p \| u \|_{H^{2 n_{\ell}+1}_h}^2  \right) \|u\|_{H^{2k}_h}^{\eps} \lesssim \| u\|_{H^{2k}_h}^{\frac{4k-6}{2k-1}+\eps}.
\end{align*}
We then infer from equation \eqref{modified_energy} and Proposition \ref{prop_estimate_modified_energy} that 
\[ \| \partial_t^k u (T)\|_{L^2_h}^2 - \| \partial_t^k u (0)\|_{L^2_h}^2 \lesssim  \| u \|_{L^{\infty}_T H^{2k}_h}^{\frac{4k-4}{2k-1}+\eps}+ \| u\|_{L^{\infty}_TH^{2k}_h}^{\frac{4k-6}{2k-1}+\eps}+ \| u \|_{L^{\infty}_T H^{2k}_h}^{\frac{4k-4}{2k-1}+\eps} + T^{\frac{2}{3}} \| u \|_{L^{\infty}_T H^{2k}_h}^{\frac{4k-\frac{5}{2}}{2k-1}+\eps},  \]
so using equations \eqref{equivalence_sobolev_time_derivative_equation} and \eqref{uniform_bound_H^s}, we finally get that
\[ \|u(t)\|_{H^{2k}_h}^2 - \|u(0)\|_{H^{2k}_h}^2 \lesssim \| \partial_t^k u(t) \|_{L^2_h}^2 + \| u \|_{H^{2k-1}_h}^2 \lesssim \| u \|_{L^{\infty}_T H^{2k}_h}^{\frac{4k-4}{2k-1}+\eps} + T^{\frac{2}{3}} \| u \|_{L^{\infty}_T H^{2k}_h}^{\frac{4k-\frac{5}{2}}{2k-1}+\eps},\]
which gives the result.
\end{proof}

\subsection{Odd Sobolev norms} We now suppose that $m=2k+1$ with $k \in \N^*$. The modified energy $\mathcal{E}_{2k+1}(u)$ is now defined for all $t\geq 0$ by
\begin{multline} \label{modified_energy_odd}
\mathcal{E}_{2k+1}(u(t))= \frac12 \| \partial_t^k \nabla^+_h u \|_{L^2_h}^2 + \frac12 \sum_{a \in h\Z^d} |u(t,a)|^{p-1} |\partial_t^k u(t,a) |^2 + \Re \sum_{n=1}^{k-1} c_n \langle \partial_t^n u\ \partial_t^{k-n} (|u|^{p-1}),\partial_t^k u \rangle_h \\
 + \frac{p-1}{8} \left\langle |u|^{p-3} , \left| \partial_t^k(|u|^2) \right|^2  \right\rangle_h + \sum_{n=1}^{k-1} c_n \langle \partial_t^{k-n} (|u|^{p-3}) \partial_t^n(|u|^2),\partial_t^k(|u|^2) \rangle_h
\end{multline}
As in the even case, differentiating $\mathcal{E}_{2k+1}(u(t))$ with respect to time, we have the following result:

\begin{proposition}
Let $u$ be a solution to \eqref{DNLS} with $u_0 \in H^{2k+1}(h\Z^d)$ and parameters $(\lambda,d,p)$ satisfying~\eqref{set_parameters}, then
\begin{multline*}
\frac{\dd}{\dd t} \mathcal{E}_{2k+1}(u(t))=  \Re \sum_{n=1}^{k-1} c_n \langle \partial_t^{n+1} u \ \partial_t^{k-n}(|u|^{p-1}),\partial_t^k u \rangle_h + \Re \sum_{n=1}^{k-1} c_n \langle \partial_t^n u \ \partial_t^{k-n+1}(|u|^{p-1}), \partial_t^k u \rangle_h \\
+ \frac12 \langle \partial_t (|u|^{p-1}), |\partial_t^k u|^2  \rangle_h + \frac{p-1}{8} \left\langle \partial_t(|u|^{p-3}) , \left| \partial_t^k(|u|^2) \right|^2  \right\rangle_h +  \sum_{n=1}^k c_n \langle \partial_t^k(|u|^{p-1}) \partial_t^n u , \partial_t^{k+1-n} u \rangle_h\\ 
+ \sum_{n=1}^{k-1} c_n \langle \partial_t^{k-n+1}(|u|^{p-3}) \partial_t^n(|u|^2),\partial_t^k(|u|^2)  \rangle_h +\sum_{n=1}^{k-1} c_n \langle \partial_t^{k-n}(|u|^{p-3}) \partial_t^{n+1}(|u|^2),\partial_t^k(|u|^2)  \rangle_h.
\end{multline*}
\end{proposition}
\begin{proof}
From integration by parts and using the expression of equation \eqref{DNLS} we get that
\[ \frac12 \frac{\dd}{\dd t} \| \partial_t^k \nabla_h^+ u \|_{L^2_h}^2 = -\Re \langle \partial_t^{k+1}u,\partial_t^k(|u|^{p-1}u) \rangle_h,   \]
so that expanding the time derivatives we infer 
\begin{multline*}
\frac12 \frac{\dd}{\dd t} \| \partial_t^k \nabla_h^+ u \|_{L^2_h}^2 = - \Re \langle\partial_t^k(|u|^{p-1})u,\partial_t^{k+1}u \rangle_h-\Re \langle |u|^{p-1} \partial_t^k u,\partial_t^{k+1} u \rangle_h  \\ 
- \Re \sum_{n=1}^{k-1} c_n \langle \partial_t^n u\ \partial_t^{k-n}(|u|^{p-1}),\partial_t^{k+1} u \rangle_h \\
=- \Re\langle \partial_t^k(|u|^{p-1})u,\partial_t^{k+1}u \rangle_h - \frac12 \frac{\dd}{\dd t} \left( \sum_{a \in h\Z^d} |u|^{p-1} |\partial_t^k u |^2 \right) + \frac12 \langle \partial_t (|u|^{p-1}), |\partial_t^k u|^2  \rangle_h\\
- \frac{\dd}{\dd t} \left( \Re \sum_{n=1}^{k-1} c_n \langle \partial_t^n u\ \partial_t^{k-n} (|u|^{p-1}),\partial_t^k u \rangle_h \right) + \Re \sum_{n=1}^{k-1} c_n \langle \partial_t^{n+1} u \ \partial_t^{k-n}(|u|^{p-1}),\partial_t^k u \rangle_h \\
+ \Re \sum_{n=1}^{k-1} c_n \langle \partial_t^n u \ \partial_t^{k-n+1}(|u|^{p-1}), \partial_t^k u \rangle_h .
\end{multline*}
We now focus on the first term in the right hand side of the above equation, noticing that as $p$ is odd we have $\partial_t(|u|^{p-1}) = \frac{p-1}{2} \partial_t(|u|^2)|u|^{p-3}$, so we can write that
\begin{multline*}
-\Re \langle \partial_t^k(|u|^{p-1})u,\partial_t^{k+1} u \rangle_h = -\frac12 \langle \partial_t^k ( |u|^{p-1}),\partial_t^{k+1} (|u|^2) \rangle_h - \sum_{n=1}^k c_n \langle \partial_t^k(|u|^{p-1}) \partial_t^n u ,\partial_t^{k+1-n} u \rangle_h \\
= - \frac{p-1}{4} \langle |u|^{p-3} \partial_t^k(|u|^2), \partial_t^{k+1}(|u|^2) \rangle_h - \sum_{n=1}^{k-1} c_n \langle \partial_t^{k-n}(|u|^{p-3}) \partial_t^n (|u|^2),\partial_t^{k+1}(|u|^2) \rangle_h \\
- \sum_{n=1}^k c_n \langle \partial_t^k(|u|^{p-1}) \partial_t^n u ,\partial_t^{k+1-n} u \rangle_h \\
= -\frac{p-1}{8} \frac{\dd}{\dd t} \left( \left\langle |u|^{p-3} , \left| \partial_t^k(|u|^2) \right|^2  \right\rangle_h \right) + \frac{p-1}{8} \left\langle \partial_t(|u|^{p-3}) , \left| \partial_t^k(|u|^2) \right|^2  \right\rangle_h \\
+ \frac{\dd}{\dd t} \left( \sum_{n=1}^{k-1} c_n \langle \partial_t^{k-n} (|u|^{p-3}) \partial_t^n(|u|^2),\partial_t^k(|u|^2) \rangle_h \right) + \sum_{n=1}^{k-1} c_n \langle \partial_t^{k-n+1}(|u|^{p-3}) \partial_t^n(|u|^2),\partial_t^k(|u|^2)  \rangle_h \\
+\sum_{n=1}^{k-1} c_n \langle \partial_t^{k-n}(|u|^{p-3}) \partial_t^{n+1}(|u|^2),\partial_t^k(|u|^2)  \rangle_h +  \sum_{n=1}^k c_n \langle \partial_t^k(|u|^{p-1}) \partial_t^n u , \partial_t^{k+1-n} u \rangle_h,
\end{multline*}
which ends the proof.
\end{proof}

The proof of the estimate
\begin{equation} \label{control_H_2k+1_norm_equation}
\| u(T) \|_{H^{2k+1}_h}^2-\| u(0) \|_{H^{2k+1}_h}^2 \leq \| u \|_{L^{\infty}_T H^{2k+1}_h}^{\frac{4k-2}{2k}+\eps} + T^{\frac{2}{3}} \| u \|_{L^{\infty}_T H^{2k+1}_h}^{\frac{4k-\frac12}{2k}+\eps},
\end{equation} 
which stands as the exact analog of equation \eqref{control_H_2k_norm_equation} in the odd case, follows the same lines as its even counterpart: taking advantage of the estimate
\[   \| \partial_t^k \nabla^+_h u - i^k\Delta_h^k \nabla^+_h u\|_{H^s_h} \leq C \|u \|_{H^{s+2k}_h}, \quad s\geq 0, \]
as well as the nonlinear Strichartz estimate \eqref{nonlinear_strichartz_estimate}, we can show that
\[\left| \mathcal{E}_{2k+1}(u(T)) \right| \lesssim \left|\mathcal{E}_{2k+1}(u(0)) \right| + \| u \|_{L^{\infty}_T H^{2k+1}_h}^{\frac{4k-2}{2k}+\eps} + T^{\frac{2}{3}} \| u \|_{L^{\infty}_T H^{2k+1}_h}^{\frac{4k-\frac12}{2k}+\eps}   \]
the same way as for the proof of Proposition \ref{prop_estimate_modified_energy}, which in turns implies equation \eqref{control_H_2k+1_norm_equation} estimating the remaining terms in equation \eqref{modified_energy_odd} as in the proof of equation \eqref{control_H_2k_norm_equation}.

\subsection{Proof of Theorem \ref{growth_discrete_sobolev_norms_theorem}}
Let $m \in \N^*$ such that $m\geq 2$, $u_0 \in H^m(h\Z^d)$ and parameters $(\lambda,d,p)$ satisfying~\eqref{set_parameters}, and let $u$ be the unique global solution of \eqref{DNLS} with $u(0,\cdot)=u_0$. We then have for all $\tau>0$ and $\eps>0$ arbitrary small that
\[ \| u(\tau) \|_{H^m_h}^2-\| u_0 \|_{H^m_h}^2 \lesssim \| u \|_{L^{\infty}_{\tau} H^m_h}^{\frac{2m-4}{m-1}+\eps} + \tau^{\frac{2}{3}} \| u \|_{L^{\infty}_{\tau} H^m_h}^{\frac{2m-\frac{5}{2}}{m-1}+\eps}.  \]
Writing $\alpha=\frac{2m-\frac{5}{2}}{m-1}+\eps<2$ for $\eps>0$ small enough, and as $\frac{2m-4}{m-1}+\eps< \alpha$, for $\tau$ small enough we get that for all $t \geq 0$,
\[ \| u(t+\tau)\|_{H^m_h}^2 \leq \| u(t)\|_{H^m_h}^2 + C\left(1+\| u\|_{L^{\infty}(\left[t,t+\tau \right];H^m(h\Z^d))}^{\alpha} \right).  \]
From energy conservation we infer 
\[ \| u(t+\tau)\|_{H^m_h}^2 \leq \| u(t)\|_{H^m_h}^2 + C\left(1+\| u\|_{H^m_h}^{\alpha} \right),   \]
so the sequence $U_n:=1+\| u(n\tau)\|_{H^m_h}^2$ satisfies $U_{n+1} \leq U_n + C U_n^{\frac{\alpha}{2}}$. As $\alpha<2$, we get by induction that $U_n \lesssim n^{\frac{4}{4-\alpha}}$, which yields to the polynomial bound
\[ \| u \|_{L^{\infty}_T H^m_h} \leq C (1+T^{2(m-1)+\eps})  \]
for all $T>0$, ending the proof of Theorem \ref{growth_discrete_sobolev_norms_theorem}.

\section{Strong convergence in Sobolev spaces} \label{convergence_section}
We have now all the tools at hand in order to prove Theorem \ref{convergence_sobolev_DNLS_theorem}. Let $\delta>d/2$, $m=\left \lceil{\delta}\right \rceil $ and $\alpha> \frac{d}{2} + m$ with the set of parameters $(\lambda,d,p)$ satisfying \eqref{set_parameters}. Let $\psi_0 \in H^{\alpha}(\R^d)$, so that for all $t \in \R$, there exists a unique solution $\psi \in \mathcal{C}(\R;H^{\alpha}(\R^d))$ of equation \eqref{NLS} such that $\psi(0)=\psi_0$. We denote $u_0= \Pi_h \psi_0$, which induces a unique solution $u$ of equation \eqref{DNLS}. In particular 
\[\|u_0\|_{H^m(h\Z^d)} \lesssim \|\psi_0\|_{H^{\alpha}(\R^d)} \]
by Lemma \ref{fourier_projection}, so $\|u_0\|_{H^m(h\Z^d)}$ is uniformly bounded with respect to $h$ . From Duhamel's formula we can write that
\[ \mathcal{S}_h u(t)= \mathcal{S}_h e^{it\Delta_h}u_0 -i\lambda \int_0^t \mathcal{S}_h e^{i(t-\tau)\Delta_h} \left( |u|^{p-1} u \right)(\tau) \dd \tau \]
and
\[  \psi(t)= e^{it\Delta}\psi_0 -i \lambda \int_0^t e^{i(t-\tau)\Delta} \left( |\psi|^{p-1} \psi \right)(\tau) \dd \tau. \]
We will then decompose our analysis on the following integrals
\begin{align*}
\| \mathcal{S}_h u (t) - \psi(t) \|_{H^s(\R^d)} & \leq \| \mathcal{S}_h e^{it\Delta_h}u_0 - e^{it\Delta} \psi_0  \|_{H^s(\R^d)}\\
 	& + |\lambda| \int_0^t \left\| \left( \mathcal{S}_h e^{i(t-\tau)\Delta_h} - e^{i(t-\tau) \Delta } \mathcal{S}_h \right) \left( |u|^{p-1} u \right)(\tau) \right\|_{H^s(\R^d)} \dd \tau \\
 	& +|\lambda| \int_0^t \left\| \mathcal{S}_h \left( |u|^{p-1} u \right)(\tau) -  \left( \left| \mathcal{S}_h u \right|^{p-1} \mathcal{S}_h u \right)(\tau)  \right\|_{H^s(\R^d)} \dd \tau \\
 	& + |\lambda| \int_0^t \left\| \left( \left| \mathcal{S}_h u \right|^{p-1} \mathcal{S}_h u \right)(\tau) -  \left(|\psi|^{p-1} \psi \right)(\tau) \right\|_{H^s(\R^d)} \dd \tau \\
 	& =: J_1(t) + J_2(t) + J_3(t) + J_4(t).
 \end{align*}

Note that the integral $J_1$ is already handle by Proposition \ref{linear_flow_prop}.

\subsection{Linear flow on the nonlinearity}
In order to estimate $J_2$, as $\Pi_h \circ \mathcal{S}_h g = g$ for any $g \in L^2(h\Z^d)$, we write that
\[
\mathcal{S}_h e^{i(t-\tau)\Delta_h}\left( |u|^{p-1} u \right) - e^{i(t-\tau) \Delta } \mathcal{S}_h\left( |u|^{p-1} u \right)  =  \mathcal{S}_h e^{i(t-\tau)\Delta_h}  \Pi_h \circ \mathcal{S}_h \left( |u|^{p-1} u \right) -  e^{i(t-\tau) \Delta }\mathcal{S}_h \left( |u|^{p-1} u \right)  \]
so applying Proposition \ref{linear_flow_prop} to $ \mathcal{S}_h \left( |u|^{p-1} u \right)$, as 
\[  \|\mathcal{S}_h \left( |u|^{p-1} u \right)(\tau) \|_{H^{\delta}(\R^d)} \lesssim \| S_h u(\tau) \|^p_{H^{\delta}(\R^d)} \lesssim \| u(\tau) \|^p_{H^{\delta}(h\Z^d)}< \infty  \]
from consecutively Corollary \ref{subordinate_shannon_prop} and Lemma \ref{boundedness_shannon}, we get that
\begin{align*}
 J_2(t)& \leq C h^{\frac{\delta-s}{2}-\frac{d}{4}} \int_0^t  (1+t-\tau) \| \mathcal{S}_h \left( |u|^{p-1} u \right)(\tau)   \|_{H^{\delta}(\R^d)} \dd \tau \\
 & \leq C h^{\frac{\delta-s}{2}-\frac{d}{4}} (1+t) \int_0^t \| u(\tau) \|_{H^{\delta}(h \Z^d)}^p \dd \tau \\
  & \leq C h^{\frac{\delta-s}{2}-\frac{d}{4}}(1+t^2)(1+t^{2(m-1)+\eps})^p (1+\| u_0 \|_{H^m(h\Z^d)})^p \\
  & \leq C h^{\frac{\delta-s}{2}-\frac{d}{4}}(1+t^{2p(m-1)+2+\eps}) (1+\|\psi_0\|_{H^{\alpha}(\R^d)})^p
 \end{align*}
 applying Theorem \ref{growth_discrete_sobolev_norms_theorem} for $\eps>0$ arbitrary small.

\subsection{Aliasing of Shannon interpolation}
We now focus on $J_3$. We will use the following property, which is a direct corollary of the bilinear estimate of Proposition \ref{aliasing_shannon_prop}:

\begin{corollary} \label{corollary_shannon}
Let $u \in H^{\delta}(h\Z^d)$ with $\delta >d/2$, and let $s \leq \delta$. Let $p=2n+1$, $n\in \N^*$, then
\[ \| \mathcal{S}_h \left( |u|^{p-1}u \right)-  \left|\mathcal{S}_h u\right|^{p-1} \mathcal{S}_h u \|_{H^s(\R^d)} \leq C h^{\delta-s} \| \mathcal{S}_h u \|_{H^{\delta}(\R^d)}^p,  \]
where $C=C(\delta,s,d,p)>0$.
\end{corollary}
\begin{proof}
Let's first recall that for all $x\in \R^d$, $\mathcal{S}_h \overline{u}(x)= \overline{\mathcal{S}_h u}(x)$. We then compute that
\begin{align*}
\| \mathcal{S}_h \left( |u|^{p-1}u \right)-  \left|\mathcal{S}_h u\right|^{p-1} \mathcal{S}_h f \|_{H^s(\R^d)} \leq & \| \mathcal{S}_h \left( |u|^{p-1}u \right)- \mathcal{S}_h \left( |u|^{p-1} \right) \mathcal{S}_h u \|_{H^s(\R^d)} \\& + \| \mathcal{S}_h \left( |u|^{p-1} \right) \mathcal{S}_h u - \left|\mathcal{S}_h u\right|^{p-1} \mathcal{S}_h u \|_{H^s(\R^d)}
\end{align*}
For the first term, we use Proposition \ref{aliasing_shannon_prop} so
\begin{align*} \| \mathcal{S}_h \left( |u|^{p-1}u \right)- \mathcal{S}_h \left( |u|^{p-1} \right) \mathcal{S}_h u \|_{H^s(\R^d)} & \leq C h^{\delta-s} \|\mathcal{S}_h u\|_{H^{\delta}(\R^d)} \| \mathcal{S}_h \left( |u|^{p-1} \right) \|_{H^{\delta}(\R^d)}  \\
& \leq C  h^{\delta-s} \|\mathcal{S}_h u\|_{H^{\delta}(\R^d)}^p,
\end{align*}
 using Corollary \ref{subordinate_shannon_prop} and the fact that $\| \mathcal{S}_h u \|_{H^s(\R^d)}=\| \mathcal{S}_h \overline{u} \|_{H^s(\R^d)}$. On the other hand, as $s < s+\delta -\frac{d}{2}$, we can bound the second term
\[ \| \mathcal{S}_h \left( |u|^{p-1} \right) \mathcal{S}_h u - \left|\mathcal{S}_h u\right|^{p-1} \mathcal{S}_h u \|_{H^s(\R^d)} \leq C \| \mathcal{S}_h u \|_{H^{\delta}(\R^d)} \| \mathcal{S}_h \left( |u|^{p-1} \right) - \left|\mathcal{S}_h u\right|^{p-1} \|_{H^s(\R^d)},     \]
using Lemma \ref{bilinear_estimate_sobolev} from Appendix \ref{appendix_section}, and we can conclude by induction as $p=2n+1$ with $n\in\N^*$.
\end{proof}

We then directly apply Corollary \ref{corollary_shannon} to get that
\[ J_3(t) \leq  C h^{\delta-s}(1+t) \sup_{\tau \in \left[0,t\right]} \| u(\tau)\|_{H^{m}(h\Z^d)}^p \leq C h^{\delta-s} (1+t^{2p(m-1)+1+\eps})(1+\|\psi_0\|_{H^{\alpha}(\R^d)})^p \]
using Theorem \ref{growth_discrete_sobolev_norms_theorem}.

\subsection{Gronwall argument}
In order to estimate $J_4$, we notice that
\[ \left| |f|^{p-1}f - |g|^{p-1}g \right| \leq C(|f|+|g|)^{p-1} |f-g| \]
from the fundamental theorem of calculus, so for $0 \leq \tau \leq t$,
\begin{align*}
\left\| \left(\left|\mathcal{S}_h u \right|^{p-1} \mathcal{S}_h u \right)(\tau) -  \left(|\psi|^{p-1} \psi \right)(\tau) \right\|_{H^s(\R^d)} & \lesssim \left\| \left( |\mathcal{S}_h u(\tau)| + |\psi(\tau)|  \right)^{p-1}  \right\|_{H^{\delta}(\R^d)} \| \mathcal{S}_h u(\tau) - \psi(\tau) \|_{H^s(\R^d)} \\
& \lesssim \left( \| \mathcal{S}_h u \|_{H^{\delta}(\R^d)}^{p-1}+\| \psi \|_{H^{\delta}(\R^d)}^{p-1} \right) \| \mathcal{S}_h u(\tau) - \psi(\tau) \|_{H^s(\R^d)}
\end{align*}
as $s<s+\delta-d/2$, hence
\[ J_4(t) \leq C\int_0^t \left( \| u (\tau) \|_{H^{\delta}(h\Z^d)}^{p-1} +\| \psi(\tau) \|_{H^{\delta}(\R^d)}^{p-1}  \right) \| \mathcal{S}_h u(\tau) - \psi(\tau) \|_{H^s(\R^d)} \dd \tau.\]
As before, we use Theorem \ref{growth_discrete_sobolev_norms_theorem} in order to bound the evolution of Sobolev norms, so that
\[ J_4(t) \leq C (1+\|\psi_0\|_{H^{\alpha}(\R^d)})^{p-1} \int_0^t (1+\tau^{2(p-1)(m-1)+\eps}) \| \mathcal{S}_h u(\tau) - \psi(\tau) \|_{H^s(\R^d)} \dd \tau. \]  
Gathering all these inequalities, on respectively $J_1$, $J_2$, $J_3$ and $J_4$, we conclude by Gronwall lemma, which direclty gives equation \eqref{convergence_sobolev_DNLS}, ending the proof of Theorem \ref{convergence_sobolev_DNLS_theorem}.

\subsection*{Acknowledgements}
The author is supported by the Labex CEMPI (ANR-11-LABX-0007-01). The author is grateful to Joackim Bernier for helping suggestions and remarks at early stage of this work.

\appendix

\section{Useful estimates} \label{appendix_section}

\subsection{Continuous inequalities} 
We first recall a classical bilinear estimate:

\begin{lemma} \label{bilinear_estimate_sobolev}
Let $f \in H^{s_1}(\R^d)$ and $g \in H^{s_2}(\R^d)$. Let $s \geq 0$ such that
\[ s\leq s_1,s_2 \quad \text{and} \quad s < s_1+s_2 - \frac{d}{2}.   \]
Then
\[ \| f g \|_{H^s(\R^d)} \leq C  \| f \|_{H^{s_1}(\R^d)} \| g \|_{H^{s_2}(\R^d)}. \]
\end{lemma}

\subsection{Discrete inequalities}
We now provide a discrete version of the Gagliardo-Nirenberg inequality, which holds for the same sets of parameters as for the continuous one. For a proof or a more complete statement, we refer to \cite{hong2019strichartz}.
\begin{lemma}
Let $0<h\leq 1$, $2 \leq q \leq \infty$ and $0< \theta <1$ such that
\[ \frac{1}{q}+ \frac{\theta s}{d}=\frac{1}{2},    \]
then
\begin{equation} \label{discrete_gagliardo_nirenberg}
\| u \|_{L^q(h\Z^d)} \leq C \| u \|_{L^2(h\Z^d)}^{1-\theta} \| u \|_{\dot{H}^s(h\Z^d)}^{\theta}
\end{equation}
for any suitable function $u$.
\end{lemma}

We also give some Strichartz estimates for the linear discrete Schrödinger equation recently proved in \cite{hong2019strichartz}, which holds uniformly with respect to the parameter $h>0$ of the discretization. 
\begin{lemma} \label{discrete_strichartz_lemma}
Let $0<h\leq 1$, and $\varphi \in L^2(h\Z^d)$. We say that the pair $(q,r)$ is discrete-Schrödinger-admissible if the following conditions hold
\begin{equation} \label{discrete_admissible_pair}
\frac{3}{q}+\frac{d}{r}=\frac{d}{2}, \quad 2\leq q,r \leq \infty,\quad (q,r,d)\neq (2,\infty,3).
\end{equation}
Then for any discrete-Schrödinger-admissible pairs $(q,r)$ and $(\tilde{q},\tilde{r})$, we have the following homogeneous Strichartz estimate
\begin{equation} \label{homogeneous_discrete_strichartz}
\| e^{i t \Delta_h} \varphi \|_{L^q(\R;L^r(h\Z^d))} \leq C \| \varphi \|_{H^{\frac{1}{q}}(h\Z^d)}
\end{equation}
and the inhomogeneous Strichartz estimate
\begin{equation} \label{inhomogeneous_discrete_strichartz}
\left\| \int_0^t e^{i (t-s) \Delta_h} F(s) \dd s \right\|_{L^q(\R;L^r(h\Z^d))} \leq C \| F \|_{\mathcal{C}(\R;H^{\frac{1}{q}}(h\Z^d) )}
\end{equation}
for any suitable function $F$.

\end{lemma}

\bibliographystyle{siam}
\bibliography{biblio}

\end{document}